\documentclass[final,12pt]{elsarticle}
\usepackage{srcltx}
\usepackage{float}
\usepackage{eurosym}
\usepackage{mathtools}
\usepackage{amsmath}
\usepackage{cases}
\usepackage{amsfonts}
\usepackage{amssymb}
\usepackage{amsthm}
\usepackage{graphicx}
\usepackage{mathrsfs}
\usepackage{xcolor}
\usepackage{xpatch}
\usepackage{exscale}
\usepackage{latexsym}
\usepackage{tikz}
\usepackage{caption}
\usepackage{lineno}
\xpatchcmd{\MaketitleBox}{\hrule}{}{}{}% remove first horizontal rule (above abstract)
\xpatchcmd{\MaketitleBox}{\hrule}{}{}{}
\usepackage[colorlinks,plainpages=true,pdfpagelabels,hypertexnames=true,colorlinks=true,pdfstartview=FitV,linkcolor=blue,citecolor=red,urlcolor=black]{hyperref}
\PassOptionsToPackage{unicode}{hyperref}
\PassOptionsToPackage{naturalnames}{hyperref}
\usepackage{enumerate}
\usepackage[shortlabels]{enumitem}
\usepackage{bookmark}
\usepackage{wasysym}
\usepackage{esint}
\usepackage[ddmmyyyy]{datetime}
\usepackage[margin=2cm]{geometry}
\makeatletter
\g@addto@macro\normalsize{%
  \setlength\abovedisplayskip{4pt}
  \setlength\belowdisplayskip{4pt}
  \setlength\abovedisplayshortskip{4pt}
  \setlength\belowdisplayshortskip{4pt}
}
\numberwithin{equation}{section}
\everymath{\displaystyle}
\usepackage[capitalize,nameinlink]{cleveref}
\crefname{section}{Section}{Sections}
\crefname{subsection}{Subsection}{Subsections}
\crefname{condition}{Condition}{Conditions}
\crefname{hypothesis}{Hypothesis}{Conditions}
\crefname{assumption}{Assumption}{Assumptions}
\crefname{lemma}{Lemma}{Lemmas}
\crefname{definition}{Definition}{Definitions}

\crefformat{equation}{\textup{#2(#1)#3}}
\crefrangeformat{equation}{\textup{#3(#1)#4--#5(#2)#6}}
\crefmultiformat{equation}{\textup{#2(#1)#3}}{ and \textup{#2(#1)#3}}
{, \textup{#2(#1)#3}}{, and \textup{#2(#1)#3}}
\crefrangemultiformat{equation}{\textup{#3(#1)#4--#5(#2)#6}}%
{ and \textup{#3(#1)#4--#5(#2)#6}}{, \textup{#3(#1)#4--#5(#2)#6}}%
{, and \textup{#3(#1)#4--#5(#2)#6}}

\Crefformat{equation}{#2Equation~\textup{(#1)}#3}
\Crefrangeformat{equation}{Equations~\textup{#3(#1)#4--#5(#2)#6}}
\Crefmultiformat{equation}{Equations~\textup{#2(#1)#3}}{ and \textup{#2(#1)#3}}
{, \textup{#2(#1)#3}}{, and \textup{#2(#1)#3}}
\Crefrangemultiformat{equation}{Equations~\textup{#3(#1)#4--#5(#2)#6}}%
{ and \textup{#3(#1)#4--#5(#2)#6}}{, \textup{#3(#1)#4--#5(#2)#6}}%
{, and \textup{#3(#1)#4--#5(#2)#6}}

\crefdefaultlabelformat{#2\textup{#1}#3}
\numberwithin{equation}{section}

\newtheorem{theorem} {Theorem}[section]
\newtheorem{proposition} {Proposition}[section]
\newtheorem{lemma}{Lemma}[section]

\newtheorem{counter example}{Counter Example}[section]
\newtheorem{remark} {Remark}[section]
\newtheorem{definition} {Definition}[section]
\def\CC{{\rm \kern.24em \vrule width.02em height1.4ex depth-.05ex \kern-.26emC}}

\def\TagOnRight

\def\AA{{it I} \hskip-3pt{\tt A}}

\def\QQ{\rlap {\raise 0.4ex \hbox{$\scriptscriptstyle |$}} {\hskip -0.1em Q}}

\makeatletter
\newcommand{\vo}{\vec{o}\@ifnextchar{^}{\,}{}}
\makeatother
\def\YYint#1#2#3{{\setbox0=\hbox{$#1{#2#3}{\iint}$}
    \vcenter{\hbox{$#2#3$}}\kern-.50\wd0}}
\def\XXint#1#2#3{{\setbox0=\hbox{$#1{#2#3}{\int}$}
    \vcenter{\hbox{$#2#3$}}\kern-.50\wd0}}

\makeatletter
\def\namedlabel#1#2{\begingroup
   \def\@currentlabel{#2}%
   \label{#1}\endgroup
}
\makeatother
\makeatletter
\newcommand{\rmh}[1]{\mathpalette{\raisem@th{#1}}}
\newcommand{\raisem@th}[3]{\hspace*{-1pt}\raisebox{#1}{$#2#3$}}
\makeatother

\newcommand{\descitem}[2]{\item[#1] \label{#2}}
\newcommand{\descref}[2]{\hyperref[#1]{\textnormal{\textcolor{black}{(}\textcolor{blue}{\bf #2}\textcolor{black}{)}}}}
\newcommand{\dref}[2]{\hyperref[#1]{\textcolor{black}{(}\textcolor{blue}{\bf #2}\textcolor{black}{)}}}
\newcommand{\be} {\begin{eqnarray}}
\newcommand{\ee} {\end{eqnarray}}
\newcommand{\Bea} {\begin{eqnarray*}}
\newcommand{\Eea} {\end{eqnarray*}}
\newcommand{\rr}{\rightarrow}

\newcommand{\de} {\delta}

\newcommand{\p}  {\prime}
\newcommand{\e}  {\epsilon}
\newcommand{\De} {\Delta}

\newcommand{\la} {\lambda}

\newcommand{\La} {\Lambda}

\newcommand{\f}{\infty}

\newcommand{\R}{\mathbb{R}}

\newcommand{\noi} {\noindent}

\newcommand{\sgn}{\mathop\mathrm{sgn}}
\newcommand{\al}{\alpha}

\newcommand{\pa}{\partial}

\newcommand{\norm}[1]{\left|\hspace{-0.2mm}\left| #1 \right|\hspace{-0.2mm}\right|}
\newcommand{\abs}[1]{\left| #1\right|}

\newcounter{whitney}
\refstepcounter{whitney}

\newcounter{ineqcounter}
\refstepcounter{ineqcounter}
\makeatletter
\def\ps@pprintTitle{%
\let\@oddhead\@empty
\let\@evenhead\@empty
\def\@oddfoot{}%
\let\@evenfoot\@oddfoot}
\makeatother
\usepackage[doublespacing]{setspace}
\usepackage[titletoc,toc,page]{appendix}

\usepackage{pgf,tikz}
\usetikzlibrary{arrows}
\usetikzlibrary{decorations.pathreplacing}
\begin{document}

\begin{frontmatter}

\title{Well-posedness  for conservation laws with spatial heterogeneities and a study of BV regularity}

\author[myaddress1]{Shyam Sundar Ghoshal}
\ead{ghoshal@tifrbng.res.in}
%\tnotetext[thankssecondauthor]{Supported in part by Inspire Research Grant.}
%\address[myaddress1]{Centre for Applicable Mathematics,Tata Institute of Fundamental Research, Post Bag No 6503, Sharadanagar, Bangalore - 560065, India.}

% \cortext[mycorrespondingauthor]{Corresponding author}

\address[myaddress1]{Centre for Applicable Mathematics,Tata Institute of Fundamental Research, Post Bag No 6503, Sharadanagar, Bangalore - 560065, India.}

\author[myaddress2]{John D. Towers}
% \cortext[mycorrespondingauthor]{Corresponding author}
\ead{john.towers@cox.net}

\author[myaddress1]{Ganesh Vaidya}
\ead{ganesh@tifrbng.res.in}

\address[myaddress2]
{MiraCosta College, 3333 Manchester Avenue, Cardiff-by-the-Sea, CA 92007-1516, USA.}

\begin{abstract}
	In this article, we consider scalar conservation laws with fluxes having spatial discontinuities and possible flat regions and study the following three aspects: (i) existence, (ii) uniqueness and (iii) BV regularity of solutions. We propose a uniqueness condition and prove existence of a weak solution via the method of wave front tracking. In the later part of the article, a BV bound of the solution is achieved under a suitable condition on the initial data and flux. We construct two counterexamples showing  BV blow-up of the solution which proves the optimality on the assumptions.
\end{abstract}
\end{frontmatter}
\tableofcontents
\section{Introduction}\label{section_intro}
This article deals with existence, uniqueness and regularity aspects of solutions for the following initial value problem,%Consider the IVP,
\begin{eqnarray}%\label{1}
u_t+A(x,u)_x&=& 0\quad\quad\quad \text{for}\,\,\,(t,x) \in(0,\infty)\times\R,\label{eq:discont}\\
\label{2}u(x,0)&=& u_0(x)\,\quad\text{for}\,\,\, x\in \mathbb{R}.\label{eq:data}
\end{eqnarray}
Here the flux $A(x,u)$ has spatial
discontinuities, possibly infinitely many, and possibly having accumulation points. In this article, we consider the case of a flux $A$ such that the minima of $u\mapsto A(x,u)$ may occur on some interval $I=[a,b]$ for $a\leq b$. Henceforth, we refer to this type of flux as a `degenerate flux'
if $a<b$, and the interval $[a,b]$ is a `flat region' (see figure \ref{figure:linear} below). %Unlike the previous works \cite{AJG,AudussePerthame,GJT_2019,KT}, the fluxes considered here need not be unimodal and can have degeneracies.
%
%
%
%which is a scalar one-dimensional conservation
%law with the flux $A(x,u)$ whose set of spatial
%discontinuities is possibly infinite with accumulation points. Unlike the previous works, the fluxes considered here need not be unimodal and can have degeneracies.

  The equations of the type \eqref{eq:discont}--\eqref{eq:data} with finitely many spatial discontinuities have been of interest for some decades now and they play an important role in mathematical modelling of various physical phenomena. % {\color{red}When $f$ and $g$ are convex (or concave), these equations play important role as they serve in mathematical modelling of various physical phenomena. } 
To name a few, the hydrodynamic limit of interacting particle systems with discontinuous speed parameter \cite{CEK},
sedimentation \cite{diehl1996conservation}, petroleum industry and polymer flooding  \cite{shen2017uniqueness}, two phase flow in heterogeneous porous medium \cite{andreianov2013vanishing}, clarifier thickener units used in waste water treatment plants \cite{burger2006extended}, traffic flow on highways with changing surface conditions \cite{garavello2007conservation}. For fluxes having finitely many spatial discontinuities, the uniqueness and existence of solutions have been studied \cite{AJG,AMV,mishra2005,mishra2007convergence,adimurthi2000conservation,AKR,BGKT,bkrt2,BKT,diehl1996conservation,garavello2007conservation,karlsen2004convergence,KT,shen2017uniqueness,towers_disc_flux_1}  via various methods, including Hamilton-Jacobi, numerical schemes, vanishing viscosity, etc.

 In this article we address three aspects of \eqref{eq:discont}: (i) uniqueness in the presence of a degenerate flux, 
 (ii) existence for a general flux which may include degeneracy and (iii) BV regularity of the entropy solution. Previously, the uniqueness problem has been dealt with in \cite{AudussePerthame,Panov2009a} when $u\mapsto A(x,u)$ is unimodal. Existence of an adapted entropy solution within the setup of \cite{AudussePerthame} has been proved in \cite{CEK, GJT_2019,Panov2009a,piccoli2018general,JDT_2020} via various methods. We note that the techniques of \cite{AudussePerthame,Panov2009a} cannot be applied for degenerate fluxes. We prove uniqueness for \eqref{eq:discont} by introducing an entropy condition for a degenerate flux. In the sequel we prove existence of the adapted entropy solution for a degenerate flux via a front tracking algorithm. In the last part of this article, we prove a BV bound for the adapted entropy solution under a suitable assumption on the initial data. Unlike the case of scalar conservation laws with Lipschitz flux, the
 total variation (TV) bound of the entropy solution can blow up \cite{ADGG} for conservation laws with discontinuous flux. In contrast, note that a BV bound holds \cite{Ghoshal-JDE} when $A$ has one spatial discontinuity and the minima of both fluxes have the same height.
 It is noticed in the present article that results on BV bounds for the  case when $A$ has infinitely many spatial discontinuities is completely different from the case of finitely many spatial discontinuities. Here  we offer a sufficient condition to obtain a BV estimate for the entropy solution and construct two counter-examples to show that the conditions are optimal. 
 
 Here we discuss all three problems and compare with previous results. 
\begin{enumerate}
	\item \textit{Uniqueness for degenerate flux:} Uniqueness of the adapted entropy solution to \eqref{eq:discont} is proved in \cite{AudussePerthame,Panov2009a} when $u\mapsto A(x,u)$ has a unique minimum point. It was an open problem to prove uniqueness when the fluxes can have flat regions, since to apply the results of \cite{AudussePerthame,Panov2009a} one needs injectivity of $u\mapsto A(x,u)$ on both side of the unique minima point. In this article, we consider the case of a flux $A$ such that the minima of $u\mapsto A(x,u)$  can occur on some interval $I=[u_M^-(x),u_M^+(x)]$ for $u_M^-(x)\leq u_M^+(x)$. To the best of our knowledge, the uniqueness problem for degenerate fluxes is open even for the case $A(x,u)=H(x)f(u)+H(-x)g(u)$ (the so-called two-flux problem), where $H$ denotes the Heaviside function and $f,g$ are fluxes having suitable regularity. To prove our uniqueness result we do not need any assumption on the traces of the solution at the interface.
	
	\item\textit{Existence of solution via front tracking:} 	We establish existence of an adapted entropy solution via a front tracking algorithm.
%	We establish existence of adapted entropy solution for degenerate flux via two methods: (i) front tracking algorithm and (ii) numerical approximation. 
%	\begin{enumerate}
Previously, a front tracking approximation for the adapted entropy solution was studied \cite{piccoli2018general} when $u \mapsto A(x,u)$ is strictly increasing and convex. By a difference scheme, existence of the adapted entropy solution is proved \cite{JDT_2020} for the case of a strictly increasing flux. Via 
the method of measured-valued solution existence of an adapted entropy solution has been established in \cite{Panov2009a}. Recently, in \cite{GJT_2019, GTV-2020} existence of an adapted entropy solution has been proved in the setup of \cite{AudussePerthame,Panov2009a} via a Godunov-type approximation. Therefore, even in the set up of \cite{AudussePerthame,Panov2009a}, the question of existence via front tracking approximation was open. We note that when $u\mapsto A(x,u)$ has critical points, a BV bound of the entropy solution may not be possible \cite{ADGG,Ghoshal-JDE} in general, even with BV initial data. One way to get compactness for \eqref{eq:discont} is to look for a TV bound of the solution transformed under the so-called singular mapping (see \cite{GJT_2019}). If $u\mapsto A(x,u)$ has one minima point, then the corresponding singular map becomes a strictly increasing function which allows one to invert the singular mapping and get back the solution. For degenerate fluxes, we note that the singular map fails to become monotone on the flat portion of the flux and hence is not invertible. We modify the singular map by adding a function which is strictly increasing on the flat potion of the flux. Unlike the case of convex conservation laws, for \eqref{eq:discont} the number of fronts may increase after interactions. Recently, it has been shown \cite{AG} that rarefactions cannot occur if the fluxes do not have any affine part. Since we work with a general convex flux which may contain an affine part, the situation becomes more complex as the number of fronts increases due to rarefactions that originate from the interface.

		%	\item\textit{Existence via numerical approximation:} We propose a Godunov type scheme for existence of adapted entropy solution to \eqref{eq:discont} with degenerate flux. Previously, convergence of Godunov approximation has been established for two-fluxes with bell shape \cite{AJG}, viscosity solution \cite{KT}. For fluxes with infinitely many discontinuities, existence of adapted entropy solution has been established via difference scheme when flux are increasing and via Godunov type scheme when fluxes are convex with unique minima. Note that while proving convergence of Godunov type approximation for degenerate flux, invertibility of singular map is an issue as we mentioned in front tracking approximation. 
%	\end{enumerate} 

	\item \textit{BV bound for uniformly convex flux:} As we mentioned before, unlike the case of classical scalar conservation laws, the TV of the entropy solution can blow up \cite{ADGG} at a finite time even if the initial data is in BV. In contrast to this, a BV bound of the entropy solution has been established \cite{Ghoshal-JDE} when $A$ has finitely many spatial discontinuities and the fluxes are $C^2$ with same height minima. For $A$ with infinitely many discontinuities BV estimates of the adapted entropy solution have been established \cite{piccoli2018general,JDT_2020} when the map $u\mapsto A(x,u)$ does not have a critical point. Hence, it is an open question to achieve a BV bound for fluxes having critical points in the $u$-variable and infinitely many spatial discontinuities.  In this article we provide a sufficient condition on the initial data to obtain a BV bound for the entropy solution for uniformly convex fluxes. That this result is optimal is demonstrated via two examples provided in sections \ref{sec:counter-ex-1} and \ref{sec:counter-ex-2}.%For fluxes having infiniteFirst, we show BV bound of entropy solution when $u\mapsto A(x,u)$ are uniformly convex. We prove this result by using wave front tracking approximation. We offer another result on BV estimate of entropy solution when flux has a special form $A(x,u) = g(\beta(x,u))$ where $\B$ has Lipschitz inverse. This result is established via numerical methods. 
	
%	{Another interesting features of this paper is that we are able to give a non trivial counter example to show that even with the same height one may not be able to a $Bv_{loc}$ bound for the solution  even when the $u_0$ is in BV. This results tells us that the work in \cite{Ghoshal-JDE} is kind of optimal in the sense that, they prove the solution for the discontinuous flux is in BV when the fluxes have the same height, that theory works only for the finitely many spatial jumps of the flux and failed to be true when there are infinitely many interfaces converging to an accumuation points. On the other hand section \ref{sec:BV} tells us some special case when one can expect a BV bound for the solution. It is to be noted that Picolli has not considered the critical points and proved a bv bound for the soltuion where as sections \ref{sec:BV} results also considered critical points for the fluxes. 	} 
	
%	{This example we construct such a way that the characteristics do not intersects before time $T=1$, which helps us to calculate the BV semi norm. }
	%\item \textit{BV bound for uniformly convex flux:}
	
	\item\textit{BV blow-up:} It has been proved by explicit examples \cite{ADGG,Ghoshal-JDE} that if the minima of convex fluxes are not the same then the TV of the entropy solution blows up even for BV initial data. On the contrary, in \cite{Ghoshal-JDE} a BV bound of the entropy solution has been achieved when $A(x,u)$ has finitely many discontinuities and $u\mapsto A(x,u)$ is $C^2$ having the same height minima. Note that the fluxes considered in the current article have the same height minima and are $C^2$. By providing two examples we show that the BV-regularity results of \cite{Ghoshal-JDE} can not hold when the flux $A$ has infinitely many discontinuities. We show this in two examples constructed in sections \ref{sec:counter-ex-1} and \ref{sec:counter-ex-2}. The example in section \ref{sec:counter-ex-2} further shows that even if we assume uniform convexity of the fluxes, the TV of the entropy solution can blow up. Construction of the example in section \ref{sec:counter-ex-2} is novel and it contains new ideas which differ from previous constructions \cite{ADGG,Ghoshal-JDE}. %This fails to if we omit one of the following assumption (i) $\pa_{uu}A(x,u)\in [a,b]$ for all $x\in \R$ and (ii)   fluxes with We note that previously mentioned results on BV bounds are obtained with some restrictive condition. It is natural to ask whether this can be relaxed or not. In section \ref{} we give an answer to this question by constructing an example in which TV of entropy solution blows up at finite time though initial data is in BV space. We consider fluxes which are not uniformly convex but still satisfies the assumptions to get existence and uniqueness. Note that it proves optimality of various results: (i) if fluxes does not satisfy $A(x,u) = g(\beta(x,u))$, uniform convexity is needed, (ii) fluxes with same height minimum value can give BV blow up if number of discontinuities is infinite which proves optimality of \cite{Ghoshal-JDE}.
\end{enumerate}

\subsection{Adapted Entropy Solutions}
%We start with listing the  assumptions on the flux $A(x,u)$ in the adapted entropy framework \cite{AudussePerthame}:
%\begin{enumerate}[label=\textbf{A\arabic*}]
%\item \label{A1} $A(x,u)$ is continuous on $\mathbb{R}\setminus {\Omega} \times \mathbb{R},$ where $\Omega$ is a closed zero measure set. Moreover, there exists a locally bounded function $u\mapsto C(u)$ such that
%\begin{equation*}
%|A(u_1,x)-A(u_2,x)|\leq C(r)|u_1-u_2|\mbox{ for }u_1,u_2\in[-r,r].
%\end{equation*}
%%where the constant $C=C(r)$ is independent of  $x$.%$A(x,u)$ is continuous on $\mathbb{R}\setminus {\Omega} \times \mathbb{R},$ where $\Omega$ is a closed zero measure set.
%\item \label{A2} There exists non negative convex functions with one critical point $B_1$ and $B_2$ satisfying, $B_1(u)\leq|A(x,u)|\leq B_2(u)$  and $\lim\limits_{|u| \rightarrow \infty}B_1(u)=\infty$.
%\item \label{A3}There exists a function $u_M:\mathbb{R} \rightarrow \mathbb{R}$ such that $A(x,\cdot)$ is injective in $(-\infty,u_M(x)]$ and $[u_M(x)-\infty)$ satisfying $A(x,u_M(x))=0$.
%\end{enumerate}
We make the following assumptions on the flux $A(x,u)$,
\begin{enumerate}[label=\textbf{A-\arabic*}]
	\item \label{A1}$A(x,u)$ is continuous on $\mathbb{R}\setminus {\Omega} \times \mathbb{R},$ where $\Omega$ is a closed zero measure set. 
	
	\item There exists a locally bounded function $q:\R\rr\R$ such that
	\begin{equation}
	\abs{A(x,u)-A(x,v)}\leq q(M)\abs{u-v}\mbox{ for a.e. }x\in\R \mbox{ and }u\in[-M,M]\mbox{ with }M>0.
	\end{equation}%For a.e. $x\in\R$, $u\mapsto A(x,u)$ is a locally Lipschitz function.

	\item \label{A3}There exist  functions $u_M^{\pm}:\mathbb{R} \rightarrow \mathbb{R}$ which are continuous on $\R \setminus \Omega,$ such that $u_M^-(x)\leq u_M^+(x)$ for $x\in\R\setminus\Omega$ and $A(x,\cdot)$ is decreasing on $(-\infty,u_M^-(x)]$ and increasing on $[u_M^+(x),\infty)$ satisfying $A(x,z)=0$ for all $u_M^-(x) \leq  z \leq u_M^+(x)$.
\end{enumerate}
%We now briefly recall some definitions and results:\\
For $\al\geq0$, let $k^{\pm}_\al$ be defined as follows,
\begin{eqnarray}
k^{-}_\al(x)\leq u_M^-(x)&\mbox{ such that }A(x,k^{-}_\al(x))=\al,\\
k^{+}_\al(x)\geq u_M^+(x)&\mbox{ such that }A(x,k^{+}_\al(x))=\al.
\end{eqnarray}

\begin{definition}\label{def:states}
	A function $k: \R \rightarrow \R$ is said to be a stationary state, if $u(t,x)=k(x)$ is the weak solution to the IVP \eqref{eq:discont}--\eqref{eq:data}, with $u_0(x)=k(x)$. For $\al>0$, we work with two types of stationary states $k_\al^{+}:\R\rr(u_M^+,\f)$ and $k_\al^{-}:\R\rr(-\f,u_M^-)$ such that $A(x,k^{\pm}_\al(x))=\al$. We define $\mathscr{S}_\al$ to be the set of all stationary states corresponding to height $\al\geq0$.
\end{definition}
\begin{remark}
	Note that for $\al=0$ there are infinitely many choices for stationary states $k(x)$. Any $k(x)$ satisfying $k(x)\in[u_M^-(x),u_M^+(x)]$ for all $x \in \R$ is a stationary state. We observe that
	%	If $k:\R \rightarrow \R$ is a stationary state, then for a.e. $x \in \R,$ we have $A(x,k(x))=constant.$ 
	%	Given an $\alpha,$ there are infinitely many choices of $k_{\alpha}.$ Note that 
	$u_M^{\pm}$ can be written as 
	\begin{equation}
	u^-_M(x)=\inf\{u\in \R;\, A(x,u)=0\} \mbox{ and } u^+_M(x)=\sup\{u\in \R;\, A(x,u)=0\}.
	\end{equation}
	For notational brevity we denote a stationary state by $k_\al(x)$ for $\al\geq0$. When $\al>0$, $k_\al$ coincides with one of $k_\al^{\pm}$.
	%We denote the stationary states with $A(x,k(x))=\alpha$ by $k_{\alpha}.$ A stationary state $k_{\alpha}$ with  $k_{\alpha}<u_M^-$ and $k_{\alpha}>u_M^+$  is denoted by $k_{\alpha}^-$ and $k_{\alpha}^+$ respectively.
\end{remark} 
%
%\begin{definition}[Adapted Entropy Condition] \label{def_adapted_entropy}
%	Let $Q= [0,T]\times\R$. 
%	\begin{eqnarray}\label{E1}
%	{\partial_t} |u(t,x)-k_{\alpha}(x)| +{\partial_x}\left[ \sgn (u-k_{\alpha}(x)) (A(u,x)-\alpha) \right] \leq 0,
%	\end{eqnarray}for $\alpha \geq 0.$
%	Or equivalently, for all $0\leq\phi \in C_c^{\infty}(Q),$
%	\begin{align}
%	&\int\limits_Q |u(t,x)-k_{\alpha}(x)|\phi_t(t,x)+\sgn (u(t,x)-k_{\alpha}(x)) (A(x,u(t,x))-\alpha)\phi_x(t,x)\, dx dt\nonumber \\
%	&+\int\limits_{\R}|u_0(x)-k_{\alpha}(x)|\phi(0,x)\,dx\geq 0,\label{E2}
%	\end{align}
%	where $k_{\alpha}: \R \rightarrow \R^+$ is a stationary state as in Definition \ref{def:states}.
%\end{definition}
%

\begin{definition}[Adapted Entropy Condition] \label{def_adapted_entropy}%{adapated}
	A function $u\in L^{\infty}(\R \times \R^+) \cap C([0,T],L_{loc}^1(\R))$ is an adapted entropy solution of the Cauchy problem  if it satisfies the following inequality in the sense of distribution:
	\begin{equation}\label{ineq:adapted}
	{\partial_t} |u(x,t)-k^{\pm}_{\alpha}(x)| +{\partial_x}\left[ \sgn (u-k^{\pm}_{\alpha}(x)) (A(u,x)-\alpha) \right] \leq 0,
	\end{equation}for $\alpha \geq 0.$
	Or equivalently, for all $0\leq\phi \in C_c^{\infty}(\R \times \R^+)$
		\begin{align}
	&\int\limits_Q |u(t,x)-k_{\alpha}(x)|\phi_t(t,x)+\sgn (u(t,x)-k_{\alpha}(x)) (A(x,u(t,x))-\alpha)\phi_x(t,x)\, dx dt\nonumber \\
	&+\int\limits_{\R}|u_0(x)-k_{\alpha}(x)|\phi(0,x)\,dx\geq 0.\label{E2}
	\end{align}
%	\begin{eqnarray}
%	\int\limits_{\R}\int\limits_{\R^+}\Big( |u(x,t)-k^{\pm}_{\alpha}(x)|\phi_t(x,t)+\left[ \sgn (u(x,t)-k^{\pm}_{\alpha}(x)) (A(x,u)-\alpha)\right]\phi_x(x,t)\Big) dx dt \geq 0.
%	\end{eqnarray}
\end{definition}

\begin{figure}
	\centering
	\begin{tikzpicture}

	\draw[thick,->] (-11 ,0) -- (0.5,0) node[anchor=north west] {$u$};
	
	\draw[color=blue] plot [smooth] coordinates { (-4,1) (-4.25,1.016) (-4.5,1.125) (-4.75,1.42) (-5,2) (-5.25,2.95) (-5.5,4.375) };
	\draw[color=blue] plot [smooth] coordinates { (-4,1)
		(-3,1) (-2.75,1.016) (-2.5,1.125) (-2.25,1.42) (-2.0,2) (-1.75,2.95) (-1.5,4.375)};
	\draw[color=blue] plot [smooth] coordinates { (-8,1) (-8.25,1.016) (-8.5,1.125) (-8.75,1.42) (-9,2) (-9.25,2.95) (-9.5,4.375) };
	\draw[color=blue] plot [smooth] coordinates { (-8,1)
		(-6,1) (-5.75,1.016) (-5.5,1.125) (-5.25,1.42) (-5.0,2) (-4.75,2.95) (-4.5,4.375)};
	%					\draw[color=blue] plot [smooth] coordinates { (-3,1)
	%						(-1,1) (-.75,1.016) (-.5,1.125) (-.25,1.42) (0,2) (.25,2.95) (.5,4.375)};
	
	\draw[dashed] (-8.2,1) -- (-8.2,0);
	\draw[dashed] (-5.8,1) -- (-5.8,0);
	
	\draw[dashed] (-4.2,1) -- (-4.2,0);
	\draw[dashed] (-2.8,1) -- (-2.8,0);
	
	\draw[thick][] (-8.3,-.1) node[anchor=north] {$u_M^-(x_1)$};
	\draw[thick][] (-5.9,-.1) node[anchor=north] {$u_M^+(x_1)$};
	\draw[thick][] (-4.3,-.1) node[anchor=north] {$u_M^-(x_2)$};
	\draw[thick][] (-2.5,-.1) node[anchor=north] {$u_M^+(x_2)$};
	
	\draw[thick][] (-7,3) node[anchor=north] {$A(x_1,u)$};
	\draw[thick][] (-3.5,3) node[anchor=north] {$A(x_2,u)$};

	\end{tikzpicture}
	\caption{This illustrates a flux $A(x,u)$ at two spatial points $x_1,x_2\in\R$. Here $A(x_i,u),\,i=1,2$ are convex functions having flat regions in $[u_M^-(x_1),u_M^+(x_2)],\,i=1,2$  respectively.}\label{figure:linear}
\end{figure}
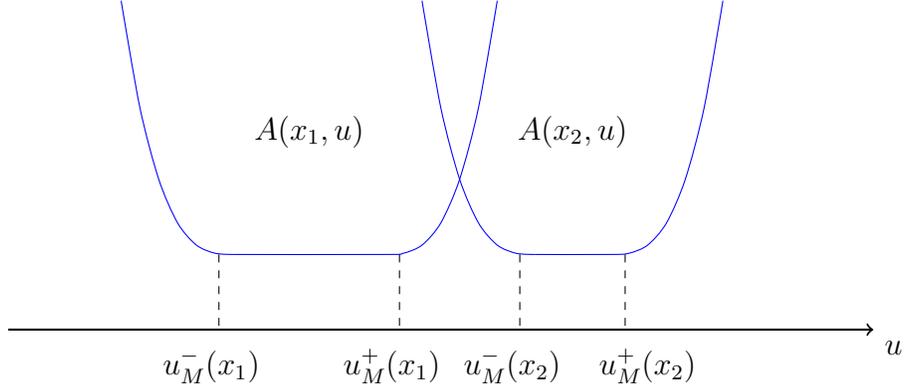
%{\bf From John: Did we ever refer to Figure~\ref{figure:linear}?}

\section{Uniqueness}
\begin{theorem}\label{theorem1}
	Let $u,v \in C(0,T;L^1_{loc}(\R)) \cap L^{\infty}(Q)$ entropic solution to the IVP \eqref{eq:discont}--\eqref{eq:data} with initial data $u_0,v_0 \in L^{\infty}(\R).$ Assume the flux satisfies the hypothesis \eqref{A1}--\eqref{A3}. Then for $t\in [0,T]$ the following holds,
	\begin{equation}\label{estimate:stability}
	\int\limits_{a}^{b}|u(t,x)-v(t,x)|dx \leq \int\limits_{a+Mt}^{b-Mt}|u_0(x)-v_0(x)|dx,
	\end{equation}
	for  $-\f\leq a<b\leq \f$ and $M:=\sup\{\abs{A_u(x,u(t,x))};\,x\in\R,0\leq t\leq T\}$.
\end{theorem}
\begin{proof}
	To prove the theorem, it is enough to show the following,
	\begin{align*}
	&\int\limits_Q |u(t,x)-v(t,x)|\phi_t(t,x)
	%\\\quad \quad \quad \quad 
	+\sgn (u(t,x)-v(t,x)) (A(x,u(t,x))-A(x,v(t,x)) )\phi_x(t,x)\, dx dt \\
	%&\quad \quad \quad \quad \quad \quad \quad \quad \quad \quad \quad \quad \quad \quad \quad \quad 
	&+\int\limits_{\R}|u_0(x)-v_0(x)|\phi(0,x)dx\geq 0.
	\end{align*}
	Let $(t,x), (s,y) \in Q.$ Corresponding to $v(s,y)$ and $u(t,x),$ we define $\tilde{v}(s,y,\cdot), \tilde{u}(t,x,\cdot): \R \rightarrow \R$ respectively as below.
	\begin{eqnarray}\label{tilv}
	\tilde{v}(s,y,x):=
	{\left\{\begin{array}{ccl}
		k^+_{A(y,v(s,y))}(x)& \, \mbox{if}&\, v(s,y)> u_M^+(y), \\[2mm]
		k^-_{A(y,v(s,y))}(x)& \, \mbox{if}&\, v(s,y)< u_M^-(y), \\[2mm]
		\min (u_M^-(x)+h,u_M^+(x)) & \, \mbox{if}&\, \text{otherwise,}
		\end{array}\right.}
	\end{eqnarray}
	where $h=v(s,y)-u_M^-(y)$. Similarly,
	\begin{eqnarray}\label{tilu}
	\tilde{u}(t,x,y):=
	{\left\{\begin{array}{ccl}
		k^+_{A(x,u(t,x))}(y)& \, \mbox{if}&\, u(t,x)> u_M^+(x), \\[2mm]
		k^-_{A(x,u(t,x))}(y)& \, \mbox{if}&\, u(t,x)<
		u_M^-(x), \\[2mm]
		\min (u_M^-(y)+h,u_M^+(y)) & \, \mbox{if}&\, \text{otherwise,}
		\end{array}\right.} 
	\end{eqnarray}
	where $h=u(t,x)-u_M^-(x)$.
	By the above definition, it is clear that 
	\begin{align*}
	A(y,\tilde{u}(t,x,y))&=A(x,u(t,x)),\\
	A(x,\tilde{v}(s,y,x))&=A(y,v(s,y)).
	\end{align*}
	Given a $(s,y)\in [0,T]\times\R,$ we rewrite the entropy condition \eqref{E2} for $u(t,x)$ with $\alpha=A(y,v(s,y))$, which implies
	\begin{equation*}
	\pa_t|u(t,x)-k_{A(y,v(s,y))}(x)|+\pa_x\left[ \sgn (u(t,x)-k_{A(y,v(s,y))}(x)) (A(x,u(t,x))-A(y,v(s,y)))\right] \le 0.
	\end{equation*}
	Set $k_{A(y,v(s,y))}(x)=\tilde{v}(s,y,x),$ by the definition of $k_{\alpha}$ we get  
	\begin{equation}\label{4.6}
	\pa_t|u(t,x)-\tilde{v}(s,y,x)|+\pa_x\left[ \sgn (u(t,x)-\tilde{v}(s,y,x)) (A(x,u(t,x))-A(y,v(s,y)))\right]\le 0.
	\end{equation}
	Similarly, now rewriting the entropy condition $v(s,y)$
	with $\alpha=A(x,u(t,x))$, we get
	\begin{eqnarray}\label{4.7}
	\pa_t |v(s,y)-\tilde{u}(t,x,y)|+\pa_x\left[ \sgn (v(s,y)-\tilde{u}(t,x,y)) (A(y,v(s,y))-A(x,u(t,x)))\right]\leq 0.
	\end{eqnarray}
	Let $\xi_{\eta},\rho_{\epsilon}\in C^{\infty}_c(\R)$  be mollifiers, such that  $supp(\rho) \subset [-2,-1].$ We define $\Phi_{\eta,\epsilon}: Q^2\rightarrow \R \in C_c^{\infty}(Q^2) $ as follows,
	\begin{equation*}
	\Phi_{\eta,\epsilon}(t,x,s,y)=\phi(t,x)\rho_{\epsilon}(t-s)\xi_{\eta}(x-y).
	\end{equation*}
	Note that 
	%\begin{enumerate}
	%    \item 
	\begin{align*}
	\frac{\partial}{\partial x}{\Phi_{\eta,\epsilon}}(t,x,s,y)&=\phi(t,x)\rho_{\epsilon}(t-s)\xi^{\p}_{\eta}(x-y)+\phi_x(t,x)\rho_{\epsilon}(t-s)\xi_{\eta}(x-y),\\
	%    \end{equation}
	%        \item
	%\begin{equation}
	\frac{\partial}{\partial y}\Phi_{\eta,\epsilon}(t,x,s,y)&=-\phi(t,x)\rho_{\epsilon}(t-s)\xi^{\p}_{\eta}(x-y),\\
	%\end{equation}
	%    \item 
	%\begin{equation}
	\frac{\partial}{\partial t}\Phi_{\eta,\epsilon}(t,x,s,y)&=\phi(t,x)\rho^{\p}_{\epsilon}(t-s)\xi_{\eta}(x-y)+\phi_t(t,x)\rho_{\epsilon}(t-s)\xi_{\eta}(x-y),\\
	%\end{equation}
	%    \item 
	%\begin{equation}
	\frac{\partial}{\partial s}{\Phi_{\eta,\epsilon}}(t,x,s,y)&=-\phi(t,x)\rho^{\p}_{\epsilon}(t-s)\xi_{\eta}(x-y).
	\end{align*}
	%\end{enumerate}
	Integrating \eqref{4.6} in $x,y,t,s$ against the function $\Phi_{\eta,\epsilon}(t,x,s,y)$,we have,
	\begin{align*}
	&\int\limits_{Q^2} |u(t,x)-\tilde{v}(s,y,x)|\phi(t,x)\rho^{\p}_{\epsilon}(t-s)\xi_{\eta}(x-y)dx dy dt ds\\
	+&\int\limits_{Q^2} |u(t,x)-\tilde{v}(s,y,x)|\phi_t(t,x)\rho_{\epsilon}(t-s)\xi_{\eta}(x-y)dx dy dt ds\\
	+&\int\limits_{Q^2}\left[ \sgn (u(t,x)-\tilde{v}(s,y,x)) (A(x,u(t,x))-A(y,v(s,y)))\right]\phi(t,x)\rho_{\epsilon}(t-s)\xi^{\p}_{\eta}(x-y)dx dy dt ds\\
	+&\int\limits_{Q^2}\left[ \sgn (u(t,x)-\tilde{v}(s,y,x)) (A(x,u(t,x))-A(y,v(s,y)))\right]\phi_x(t,x)\rho_{\epsilon}(t-s)\xi_{\eta}(x-y)) dx dy dt ds \\
	+&\int\limits_{Q}\int\limits_{\R}|u_0(x)-\tilde{{v}}(s,y,x)|\phi(0,x)\rho_{\epsilon}(-s)\xi_{\eta}(x-y) dx dy ds\geq 0.
	\end{align*}
	Integrating \eqref{4.7} in $x,y,t,s$ against function $\Phi_{\eta,\epsilon}(t,x,s,y)$, we get
	\begin{align*}
	&-\int\limits_{Q^2} |v(s,y)-\tilde{u}(t,x,y)|\phi(t,x)\rho^{\p}_{\epsilon}(t-s)\xi_{\eta}(x-y)dx dy dt ds\\
	&-\int\limits_{Q^2}\left[ \sgn (v(s,y)-\tilde{u}(t,x,y)) (A(y,v(s,y))-A(x,u(t,x)))\right]\phi(t,x)\rho_{\epsilon}(t-s)\xi^{\p}_{\eta}(x-y) dx dy dt ds\\
	+&\int\limits_{Q}\int\limits_{\R}|v_0(x)-\tilde{{u}}(t,x,y)|\phi(t,x)\rho_{\epsilon}(t)\xi_{\eta}(x-y) dx dy dt\geq 0.
	\end{align*}
	Adding the above two inequalities and collecting the common terms, we have the sum of the following 6 terms
	\begin{enumerate}
		\item 
		\begin{equation*}
		\int\limits_{Q^2} |u(t,x)-\tilde{v}(s,y,x)|\phi_t(t,x)\rho_{\epsilon}(t-s)\xi_{\eta}(x-y)dx dy dt ds,
		\end{equation*}
		\item 
		\begin{equation*}
		\int\limits_{Q^2} \left(|u(t,x)-\tilde{v}(s,y,x)|-|v(s,y)-\tilde{u}(t,x,y)|\right)\phi(t,x)\rho^{\p}_{\epsilon}(t-s)\xi_{\eta}(x-y)dx dy dt ds,
		\end{equation*}
		\item 
		\begin{equation*}
		\int\limits_{Q^2}\left[ \sgn (u(t,x)-\tilde{v}(s,y,x)) (A(x,u(t,x))-A(y,v(s,y)))\right]\phi_x(t,x)\rho_{\epsilon}(t-s)\xi_{\eta}(x-y) dx dy dt ds,
		\end{equation*}
		\item 
		\begin{align*}
		&\int\limits_{Q^2}\left[\Big(\sgn (v(s,y)-\tilde{u}(t,x,y))+\sgn (u(t,x)-\tilde{v}(s,y,x))\Big) (-A(y,v(s,y))+A(x,u(t,x)))\right]\\
		&\quad \quad\quad \quad\quad \quad\quad \quad\quad \quad\quad \quad\quad \quad \quad\quad \quad\quad \quad\quad  \quad
		\phi(t,x)\rho_{\epsilon}(t-s)\xi^{\p}_{\eta}(x-y) dx dy dt ds,
		\end{align*}
		\item 
		\begin{equation*}
		\int\limits_{Q}\int\limits_{\R}|u_0(x)-\tilde{{v}}(s,y,x)|\phi(t,x)\rho_{\epsilon}(-s)\xi_{\eta}(x-y) dx dy ds,
		\end{equation*}
		\item 
		\begin{equation*}
		\int\limits_{Q}\int\limits_{\R}|v_0(x)-\tilde{{u}}(t,x,y)|\phi(t,x)\rho_{\epsilon}(t)\xi_{\eta}(x-y) dx dy dt,
		\end{equation*}
	\end{enumerate}
	is greater than equal to 0.\\
	We now make the following claims:
	\begin{enumerate}[label=\roman*.]
		\item\label{C1} $\tilde{v}(s,y,x) \rightarrow \tilde{v}(s,y,y)=v(s,y)$ as $x \rightarrow y$ for a.e. $y\in \R$.
		\item\label{C2}  $\tilde{u}(t,x,y) \rightarrow \tilde{u}(t,x,x)=u(t,x)$ as $y \rightarrow x$ for a.e. $x\in \R$.
		\item\label{C3}$\left(\sgn (v(s,y)-\tilde{u}(t,x,y))+\sgn (u(t,x)-\tilde{v}(s,y,x))\right) (-A(y,v(s,y))+A(x,u(t,x)))=0$.
	\end{enumerate}
	
	To prove \eqref{C1}, let us assume that $A(\cdot,\cdot)$ is continuous at $(y,v(s,y))$. Therefore,
	\begin{equation*}
	A(x,\tilde{v}(s,y,y)) \rightarrow A(y,\tilde{v}(s,y,y))  \text{ as } x \rightarrow y.
	\end{equation*}
	Now, If $v(s,y)< u_M^-(y),$ then, by \eqref{tilv} we have $\tilde{v}(s,y,x)<u_M^-(x)$ for all $x\in\R$, in particular, $\tilde{v}(s,y,y)<u_M^-(y)$. By \eqref{tilv}, we also have, $A(y,\tilde{v}(s,y,y))=A(y,v(s,y))=A(x,\tilde{v}(s,y,x))$. Therefore, $A(x,\tilde{v}(s,y,y))- A(x,\tilde{v}(s,y,x)) \rr0$ as $x\rr y$. Since $\{\tilde{v}(s,y,x),x\in\R\}$ is bounded and $A(x,u)$ is monotone on $(-\f,u_M^-(x)]$, we have $\tilde{v}(s,y,x)\rr\tilde{v}(s,y,y)$ as $x\rr y$. When ${v}(s,y)>u_M^+(y)$, by a similar argument we have $\tilde{v}(s,y,x)\rr\tilde{v}(s,y,y)$ as $x\rr y$. Suppose $v(s,y)\in[u_M^{-}(y),u_M^{+}(y)]$. Then by \eqref{tilv}, we have $\tilde{v}(s,y,x)=\min \{u_M^-(x)+v(s,y)-u_M^-(y),u_M^+(x)\}$. Since $y$ is a continuity point of $u_M$, we have $u_M^-(x)+v(s,y)-u_M^-(y)\rr v(s,y)$ as $x\rr y$. Therefore, we have $\tilde{v}(s,y,x)\rr\min \{v(s,y),u_M^+(y)\}$ as $x\rr y$. Since $v(s,y)\in[u_M^{-}(y),u_M^{+}(y)]$, we have $\tilde{v}(s,y,x)\rr v(s,y)$ as $x\rr y$. This proves \eqref{C1}. By a similar argument, we prove \eqref{C2}.

	%Up to a subsequence we have $0\leq A(x_n,\tilde{v}(s,y,y))- A(x,\tilde{v}(s,y,x)) \rr0$ or $0\geq A(x_n,\tilde{v}(s,y,y))- A(x_n,\tilde{v}(s,y,x_n)) \rr0$ as $x_n\rr y$. Since $A(x,\cdot)$ is decreasing on $(-\f,u_M(x)]$, we have either $\tilde{v}(s,y,x_n)\rr\tilde$
	%$$A(x,\tilde{v}(s,y,y)) \rightarrow A(y,\tilde{v}(s,y,y))  \text{ as } x \rightarrow y$$
	%$$A(y,\tilde{v}(s,y,y))=A(y,v(s,y))=A(x,\tilde{v}(s,y,x))$$
	%which implies $$A(x,\tilde{v}(s,y,x)) \rightarrow A(y,v(s,y,y)) \text{ as } x \rightarrow y.$$
	%
	%
	%By definition of $\tilde{v},$ we have $\tilde{v}(s,y,x) < u_M^-(x).$ Now injectivity of $A(y,\cdot)$ in $(-\infty,u_M^-(y))$ implies \eqref{C1}.  If $v(s,y)> u_M^-(y),$ \eqref{C1} follows by repeating the arguments with $u_M^+$.\\
	%If $A(y,v(s,y))=0,$ then continuity of $u_M^{\pm}$ at $y$ implies \eqref{C1}.\\
	%\eqref{C2} follows similarly.\\
	To prove \eqref{C3}, consider, 
	\begin{eqnarray}\label{eqnc3}
	\left(\sgn (v(s,y)-\tilde{u}(t,x,y))+\sgn (u(t,x)-\tilde{v}(s,y,x))\right) (A(x,u(t,x))-A(y,v(s,y))).
	\end{eqnarray}
	Now, depending on the values of $A(y,v(s,y)), A(x,u(t,x))$, we will have the following four cases:
	\begin{enumerate}[label=\textbf{Case \arabic*.}]
		\item Suppose $A(y,v(s,y)), A(x,u(t,x))>0$. In this case, we follow a similar argument as in \cite{AudussePerthame}. Note that 
		\begin{align}
		sgn(u(t,x)-u_M(x))&=sgn(u(t,x,y)-u_M(y)),\label{sgn1}\\
		sgn(v(s,y)-u_M(y))&=sgn(v(s,y,x)-u_M(x)).\label{sgn2}
		\end{align} Suppose $v(s,y)>\tilde{u}(t,x,y)$. Then there are three possibilities 
		\begin{enumerate}
			\item\label{sgn:vu1} $v(s,y)>\tilde{u}(t,x,y)>u_M^+(y)$. Therefore, $A(x,\tilde{v}(s,y,x))=A(y,v(s,y))>A(y,\tilde{u}(t,x,y))=A(x,u(t,x))$. By using \eqref{sgn1} and \eqref{sgn2}, we get $\tilde{v}(s,y,x)>{u}(t,x)>u_M^+(x)$.
			
			\item $u_M^->v(s,y)>\tilde{u}(t,x,y)$. By a similar argument as in \eqref{sgn:vu1} we get $u_M^->\tilde{v}(s,y,x)>{u}(t,x)$.
			
			\item\label{sgn:vu3} $v(s,y)>u_M^+(y)>u_M^-(y)>\tilde{u}(t,x,y)$. By using \eqref{sgn1} and \eqref{sgn2}, we get $\tilde{v}(s,y,x)>u_M^+(x)>u_M^-(x)>u(t,x)$.
			
		\end{enumerate}
		Hence, the term in \eqref{eqnc3} vanishes for this case. By a similar argument, we can show the same result when $v(s,y)<\tilde{u}(t,x,y)$. Note that when $v(s,y)=\tilde{u}(t,x,y)$, we have
		\begin{equation}
		A(y,v(s,y))=A(y,\tilde{u}(t,x,y))=A(x,u(t,x))=A(x,\tilde{v}(s,y,x)).
		\end{equation}
		Therefore, the term in \eqref{eqnc3} vanishes for this case too.
		
		\item Now we consider $A(y,v(s,y))=A(x,u(t,x))=0$. This implies $(A(x,u(t,x))-A(y,v(s,y)))=0$ and hence \eqref{eqnc3} vanishes.

		\item\label{case3} Next we consider $A(y,v(s,y))=0$ and $A(x,u(t,x))>0$. In this case, using the \eqref{tilv}-\eqref{tilu}, we get $v(s,y)\in [u_M^-(y),u_M^+(y)]$ and $v(s,y,x)\in [u_M^-(x),u_M^+(x)]$. Suppose $u(t,x)>u_M^+(x)$. By \eqref{sgn2} we have $\tilde{u}(t,x,y)>u_M^+(y)$. Therefore, the term in \eqref{eqnc3} vanishes.  By a similar argument, we can show that \eqref{eqnc3} vanishes when $u(t,x)<u_M^-(x)$.
		%we have 
		%$$\Big(\sgn (v(s,y)-\tilde{u}(t,x,y))+\sgn (u(t,x)-\tilde{v}(s,y,x))\Big) =0.$$which implies \eqref{eqnc3} is zero.
		\item Finally, we consider $A(y,v(s,y))>0$ and $A(x,u(t,x))=0$. By a similar argument as in \eqref{case3} we show the term in \eqref{eqnc3} vanishes.
		%In this case, again using the \eqref{tilv}-\eqref{tilu}, we get
		%$$\Big(\sgn (v(s,y)-\tilde{u}(t,x,y))+\sgn (u(t,x)-\tilde{v}(s,y,x))\Big) =0.$$
	\end{enumerate}
	This completes proof of \eqref{C3}. Now rest of the proof runs on the similar lines  \cite{AudussePerthame}, which we briefly discuss for the sake of completeness.
	
	\begin{enumerate}[label= \textbf{Term} \arabic*: ]
		\item By using \eqref{C1} we first sending $\eta$ to zero and then sending $\epsilon$ to zero and obtain,
		\begin{align*}
		&\int\limits_{Q^2} |u(t,x)-\tilde{v}(s,y,x)|\phi_t(t,x)\rho_{\epsilon}(t-s)\xi_{\eta}(x-y)\,dx dy dt ds\\
		& \rightarrow \int\limits_{Q} |u(t,x)-{v}(t,x)|\phi_t(t,x)dx dt.
		\end{align*}
		%Follows by first sending $\eta$ to zero using \eqref{C1} and then sending $\epsilon$ to zero.
		\item Now we want to show
		\begin{equation}\label{pass-to-lim2}
		\int\limits_{Q^2} \Big(|u(t,x)-\tilde{v}(s,y,x)|-|v(s,y)-\tilde{u}(t,x,y)|\Big)\phi(t,x)\rho^{\p}_{\epsilon}(t-s)\xi_{\eta}(x-y)dx dy dt ds \rightarrow 0.
		\end{equation}
		By using the reverse triangular inequality $||a|-|b||\leq |a-b|$, we obtain,
		\begin{align*}
		||{u}(t,x)-\tilde{{v}}(s,y,x)|-|{v}(s,y)-\tilde{u}(t,x,y)||
		&\leq |{u}(t,x)-\tilde{{u}}(t,x,y)-\tilde{{v}}(s,y,x)+{v}(s,y)|\\
		&\le |{u}(t,x)-\tilde{u}(t,x,y)|+|\tilde{{v}}(s,y,x)-{v}(s,y)|.
		\end{align*}
		%By a similar argument, we prove that
		%\begin{equation*}
		%|u(t,x)-\tilde{v}(s,y,x)|-|v(s,y)-\tilde{u}(t,x,y)|\ge
		%   - \left(|{u}(t,x)-\tilde{u}(t,x,y)|+|\tilde{v}(s,y,x)-{v}(s,y)|\right).
		%    \end{equation*}
		Now, using \eqref{C1}-\eqref{C2}, first let $\eta \rightarrow 0.$ Therefore,
		\begin{equation*}
		\int\limits_{\R^2} \left(|u(t,x)-\tilde{v}(s,y,x)|-|v(s,y)-\tilde{u}(t,x,y)|\right)\phi(t,x)\xi_{\eta}(x-y)dx dy\rr 0\mbox{ as }\eta\rr0,
		\end{equation*}
		for each fixed $\e>0$. This proves \eqref{pass-to-lim2}.
		
		\item Now we wish to show
		\begin{align}
		&\int\limits_{Q^2}\left[ \sgn (u(t,x)-\tilde{v}(s,y,x)) (A(x,u(t,x))-A(y,v(s,y)))\right]\phi_x(t,x)\rho_{\epsilon}(t-s)\xi_{\eta}(x-y)) dx dy dt ds \nonumber\\
		&\rightarrow \int\limits_{Q}\left[ \sgn (u(t,x)-{v}(t,x)) (A(x,u(t,x))-A(y,v(t,x)))\right]\phi_x(t,x)) dx  dt \mbox{ as }\eta\rr0,\e\rr0.\label{pass-to-lim3}
		\end{align}
		We consider $\mathcal{A}(x,u,v)=sgn(u-v)(A(x,u)-A(x,v))$. Since $\tilde{v}(s,y,x)\in L^{\f}([0,T]\times\R\times\R)$, there exists $C>0$ such that the following holds
		\begin{equation*}
		\abs{\mathcal{A}(x,u(t,x),\tilde{v}(s,y,x))-\mathcal{A}(x,u(t,x),v(s,y))}\leq C\abs{\tilde{v}(s,y,x)-v(s,y)}.
		\end{equation*}
		By using \eqref{C1} and letting $\eta \rightarrow 0$, we have
		\begin{align*}
		&\int\limits_{\R^2}\left[ \sgn (u(t,x)-\tilde{v}(s,y,x)) (A(x,u(t,x))-A(y,v(s,y)))\right]\phi_x(t,x)\xi_{\eta}(x-y)) dx dy \\
		&\rightarrow \int\limits_{\R}\left[ \sgn (u(t,x)-{v}(t,x)) (A(x,u(t,x))-A(x,v(t,x)))\right]\phi_x(t,x)) dx .
		\end{align*}
		Now we pass to the limit as $\e\rr0$ and get \eqref{pass-to-lim3}.
		
		\item By using \eqref{C3} we get
		\begin{align*}
		&\int\limits_{Q^2}\sgn (v(s,y)-\tilde{u}(t,x,y))(-A(y,v(s,y))+A(x,u(t,x)))\phi(t,x)\rho_{\epsilon}(t-s)\xi^{\p}_{\eta}(x-y) dx dy dt ds  \\
		+&\int\limits_{Q^2}\sgn (u(t,x)-\tilde{v}(s,y,x)) (-A(y,v(s,y))+A(x,u(t,x)))\phi(t,x)\rho_{\epsilon}(t-s)\xi^{\p}_{\eta}(x-y) dx dy dt ds \\
		=&0.
		\end{align*}
		%Follows due to \eqref{C3}.
		\item By using \eqref{C1} we obtain
		\begin{equation*}
		\int\limits_{Q}\int\limits_{\R}|u_0(x)-\tilde{{v}}(s,y,x)|\phi(t,x)\rho_{\epsilon}(-s)\xi_{\eta}(x-y) dx dy ds \rightarrow \int\limits_{\R}|u_0(x)-v_0(x)|dx\mbox{ as }\e\rr0.
		\end{equation*}
		\item Since support of $\rho$ is a subset of $[-2,-1]$, we have,
		\begin{equation*}
		\int\limits_{Q}\int\limits_{\R}|v_0(x)-\tilde{{u}}(t,x,y)|\phi(t,x)\rho_{\epsilon}(t)\xi_{\eta}(x-y) dx dy dt \rightarrow 0\mbox{ as }\e\rr0.
		\end{equation*}
	\end{enumerate}
	This completes the proof of Theorem \ref{theorem1}.
	
\end{proof}

%%\color{blue}
%\section{Introduction}
%We are interested in  the following Cauchy problem:
%\begin{eqnarray}
%u_t+A(x,u)_x&=& 0\quad\quad\quad \text{for}\,\,\,(x,t) \in \mathbb{\R}\times (0,\infty),\label{eq:discont}\\
%u(x,0)&=& u_0(x)\,\quad\text{for}\,\,\, x\in \mathbb{R},\label{eq:data}
%\end{eqnarray} 
%which is a scalar one-dimensional conservation law with the flux $A(x,u)$ whose set spatial discontinuities is possibly infinite and has accumulation points.  
% 
\section{Existence via front-tracking}\label{MainResults}
%
%  \begin{theorem}[Audusse-Parthame]
%  Let $u,v \in L^{\infty}(\R \times \R^+) \cap C([0,T],L_{loc}^1(\R))$ be adapted entropy solutions, corresponding to the initial data $u_0$ and $v_0$, then for a.e. $t \in [0,T]$ we have, \\
%  \begin{equation}
%  \int\limits_a^b |u(x,t)-v(x,t)|dx \leq \int\limits_{a-Lt}^{b+Lt} |u_0(x)-v_0(x)|dx,
%  \end{equation}
%  where $L=\sup \left\{\abs{\partial_u A(x,u)}: x\in \R, |u|\leq \max\{||u_0||_{L^{\infty}},||v_0||_{L^{\infty}}\}\right\}.$
%  \end{theorem}
%   If flux does not have spatial discontinuities, the above formulation boils down to Kruzkhov entropy condition. 
%   
To use the Front tracking algorithm for $A(x,u)$ of the type \eqref{A1}-\eqref{A3}, we make the following additional assumptions:
\begin{enumerate}[label=\textbf{B-\arabic*}]
\item \label{A4}There exists a continuous function $\eta: \mathbb{R} \rightarrow \mathbb{R}$ and a BV function $a:\R \rightarrow \R$ such that $$|A(x,u)-A(y,u)| \leq \eta(u)|a(x)-a(y)| \quad \forall x,y,u \in \R.$$

\item \label{B2} The map $u\mapsto A(x,u)$ is convex for each $x\in \R$.
\item \label{A2} There exist non negative convex functions $B_1$ and $B_2$ satisfying \\$B_1(u)\leq|A(x,u)|\leq B_2(u)$  and $\lim\limits_{|u| \rightarrow \infty}B_1(u)=\infty$.
\item \label{A5}The function $u_M(\cdot) \in BV(\R)$ and $(u_M^+(\cdot)-u_M^-(\cdot))\in BV(\R)$.
\end{enumerate}
%Note that \ref{A4} is a technical assumption which helps in approximating $A(\cdot,\cdot)$ by functions with finitely many spatial discontinuity. %If, in addition, the spatial discontinuities are discrete, \ref{A4} can be skipped. Also, \ref{A4} can be weakened if the limit points of the spatial discontinuities are discrete. 
%\ref{A5} is a sufficient condition to obtain $TV$ bounds of singular map.
\begin{theorem}\label{MT}
Suppose $u_0(\cdot)\in L^\f(\R),$ and $A(\cdot,\cdot)$ satisfies the assumptions \eqref{A1}--\eqref{A3} and \eqref{A4}--\eqref{A5}. Then, there exists a unique adapted entropy solution to \eqref{eq:discont} corresponding to the initial data $u_0$. %the wave front tracking approximations converge pointwise almost  everywhere to the unique solution of Cauchy problem \eqref{eq:discont}-\eqref{eq:data} satisfying the  adapted entropy condition \eqref{ineq:adapted}. 
\end{theorem}

%\begin{theorem}\label{MT}
%	Suppose $u_0(\cdot)\in BV(\R),$ and $A(\cdot,\cdot)$ satisfies the assumptions \eqref{A1}--\eqref{A3} and \eqref{A4}--\eqref{A5}. Then, the wave front tracking approximations converge pointwise almost  everywhere to the unique solution of Cauchy problem \eqref{eq:discont}-\eqref{eq:data} satisfying the  adapted entropy condition \eqref{ineq:adapted}. 
%\end{theorem}
%\section{Preliminaries}\label{prilim}

\subsection{Results from two-flux theory}\label{sec:2flux}
We momentarily focus on the so-called two-flux theory where $A(x,u)=H(-x)g(u)+H(x)f(u)$, and
$H(\cdot)$ denotes the Heaviside function. 
Here $f$ and $g$ are a pair of functions $f,g:\R\rr\R$ such that \eqref{A1}--\eqref{A3} hold.
Note that there are constants $\theta_f^{\pm}$, $\theta_g^{\pm}$ such that
\begin{equation*}
u_M^{\pm}(x) = H(-x) \theta_g^{\pm} + H(x) \theta_f^{\pm}.
\end{equation*}
 \begin{definition}[Interface entropy functional, \cite{AMV}]
Interface entropy functional $I_{(\theta_g,\theta_f)}$ corresponding to the connection $(\theta_g,\theta_f)$ is defined by
\begin{equation}\label{ineq:IEF}
I_{(\theta_g,\theta_f)}(u_-,u_+):=\sgn(u_--\theta_g)(g(u_-)-g(\theta_g))-\sgn(u_+-\theta_f)(f(u_+)-f(\theta_f)),
\end{equation}
where $\theta_g,\theta_f$ are defined as follows
\begin{equation*}
\theta_g=\frac{\theta_g^-+\theta_g^+}{2}\mbox{ and }\theta_f=\frac{\theta_f^-+\theta_f^+}{2}.
\end{equation*}
\end{definition}
We note that the interface entropy function $I_{AB}$ has been introduced in \cite{AMV} for $A-B$ connection which can boil down to $I_{(\theta_g,\theta_f)}$ for $A=\theta_g,B=\theta_f$ where $\theta_g,\theta_f$ correspond to unique minima points of $f,g$ respectively. In our case, there can be more than one minima points for the fluxes $f,g$ but it turns out that similar entropy functional works in our set up as well, and it selects same the solution as the adapted entropy condition does. More precisely, we prove the following theorem,
  \begin{lemma} \label{io} Suppose $A(x,u)=H(-x)g(u)+H(x)f(u)$ for two function $f,g:\R\rr\R$ such that \eqref{A1}--\eqref{A3} hold. Then, for a weak solution of \eqref{eq:discont}--\eqref{eq:data} with left and right trace the following two are equivalent
  	\begin{enumerate}[label=E-\arabic*.]
  		\item\label{Case:K} It satisfies Kruzkov entropy away from interface and $I_{(\theta_g,\theta_f)} \geq 0$ at the interface.
  		
  		\item\label{Case:AP} It satisfies the adapted entropy inequality \eqref{ineq:adapted}.
  	\end{enumerate}%satisfies the adapted entropy iff it satisfies $I_{(\theta_g,\theta_f)} \geq 0.$
  \end{lemma}
 \begin{proof}
% 	We first observe the following: if a weak solution of the form $u(x,t):=(1-H(x))u_-+H(x)u_+$, then we have $f(u_+)=g(u_-)$. Let $a,b\in\R$ such that $f(b)=g(a)$. it satisfies adapted entropy iff  
% 	We first show \eqref{Case:AP} implies \eqref{Case:K}. 
% 	Since $f,g$ are convex, by Remark \ref{} we get, $u$ satisfies Kruzkov inequality away from interface. %At the interface note that
% A shock solution at the interface of spatial discontinuity satisfies adapted entropy iff $g(u_-)=f(u_+)$  and the  for all $a,b \in \R$ such that  $g(a)=f(b)$ and either $a\leq \theta_g,b\leq \theta_f$ or $a\leq \theta_g,b\leq \theta_f$ following holds,
 Fix a $t>0$. Define $u_\pm=u(0\pm,t)$. By choosing an appropriate sequence of test functions it can be shown that $f(u_+)=g(u_-)$. Since $u$ satisfies \eqref{ineq:adapted}, by a similar argument as before, we get
 \begin{equation}\label{ineq:adapted:1}
 \sgn(u_- -k_\al(0-))(g(u_-)-\alpha)-\sgn (u_+ -k_\al(0+)) (f(u_+)-\alpha)\geq 0\mbox{ for }k_\al\in\mathscr{S}_\alpha,\al\geq0.
 \end{equation}
We take a special choice $\alpha=0$ and $k_{\alpha}(0\pm)=(u_M^{-}(0\pm)+u_M^{+}(0\pm))/2$. This gives \eqref{Case:K}. To see that \eqref{Case:K} implies \eqref{Case:AP}. We only check the adapted entropy condition at the interface. Away from the interface, $u$ satisfies \eqref{Case:AP} since $u$ satisfies the classical Kruzkov entropy condition. We fix a $k_\al\in \mathscr{S}_\al$ for $\alpha\geq0$. If $\al=0$, then, it suffices to show
\begin{equation}\label{cal:ineq:lemma1}
 \left[\sgn(u_- -k_\al(0-))-\sgn (u_+ -k_\al(0+)) \right]f(u_+)\geq 0.
\end{equation}
It is clear for $f(u_+)=0$. Without loss of generality, we assume $f(u_+)>0$. Then, \eqref{cal:ineq:lemma1} follows from inequality $I_{(\theta_g,\theta_f)} \geq 0$. Next we consider the case when $\al>0$, then, it suffices to show
\begin{equation}\label{cal:ineq:lemma2}
\left[\sgn(u_- -k^+_\al(0-))-\sgn (u_+ -k^+_\al(0+)) \right](f(u_+)-f(k_\al(0+)))\geq 0.
\end{equation}
We only check for $k_\al^+$, same argument works for $k_\al^-$. Now, observe that if $f(u_+)>0$ and $u_+\geq \frac{\theta_f^-+\theta_f^+}{2}$, then we have $u_+>\theta_f^+$. Similarly, we have that if $f(u_+)>0$ and $u_+\leq \frac{\theta_f^-+\theta_f^+}{2}$, then we have $u_+<\theta_f^-$. Rest of the proof we split into cases.
\begin{enumerate}[(i)]
	\item Suppose $f(u_+)\geq \al$. Then, there arise following two subcases.
	\begin{enumerate}
		\item If $u_+\geq \theta_f^+$, then we have $u_+\geq k_\al^+(0+)$. Since $(u_-,u_+)$ satisfies \eqref{ineq:adapted:1}, we have $u_-\geq \theta_g^+$. Therefore, we have $u_-\geq k_\alpha^+(0-)$. Hence, we have \eqref{cal:ineq:lemma2}.
		\item If $u_+\leq\theta_f^-$, then $u_+< k_\al^+(0+)$. Hence, it clearly follows \eqref{cal:ineq:lemma2}. 
	\end{enumerate}  
\item Now, suppose $f(u_+)\leq \alpha$. We have following two subcases.
\begin{enumerate}
	\item If $\theta_f^+\leq u_+$, then $\leq k_\alpha^+(0+)$ and by previous observation, we have $u_-\geq \theta_g^+$. Note that $g(u_-)\leq\al$. Hence, $u_-\leq k_\al^+(0-)$. Hence, \eqref{cal:ineq:lemma2} holds. 
	\item Now, if $u_+\leq \theta_f^-$, then we have $u_-\leq \theta_g^-$. Note that in this case, $u_+\leq k_\al^+(0+)$ and $u_-\leq k_\al^+(0-)$. Therefore, \eqref{cal:ineq:lemma2} holds.
\end{enumerate} 
\end{enumerate}
  \end{proof}

%\subsection{\textbf{Riemann Solver}}
%Now we define the Reimann solver for the flux of the type $A(x,u)=H(-x)g(u)+H(x)f(u).$ For the initial data given by
%\begin{eqnarray}
%u_0(x)=\left\{\begin{array}{cc}
%   u_l
%   &   x<0 ,\\
%u_r &  x>0,
%  \end{array}  \right.
%  \end{eqnarray}
%  we have the following.
%\begin{definition}[Riemann Solver]
%A Riemann solver is a function $\textit{R}: \R \times \R \rightarrow \R \times \R$ which maps the Riemann data $(u_l,u_r)$ to the left and right traces $(u_-,u_+)$of the adapted entropy solution. 
%\end{definition}For our problem this map can be explicitly given by $\textit{R}(u_l,u_r)= (u_-,u_+)$. Where, 
%\begin{eqnarray}
%u_-=\displaystyle{g^{-1}_-\left(\max \left\{g(\max(u_l,\theta_g)),f(\min(u_r,\theta_f))\right\}\right)},\\
%u_+=\displaystyle{f^{-1}_+\left(\max \left\{g(\max(u_l,\theta_g)),f(\min(u_r,\theta_f))\right\}\right)}.
%\end{eqnarray}
%Clearly for constant initial data $u_l=u_r=c,$ with $g(c) \neq f(c),$ solution is not constant implying the increase of total variation.
%For the sake of completeness we mention the solutions of the Riemann problem under various possible $u_l$ and $u_r,$ which can also be obtained through Hopf-Lax type formula as done in \cite{adimurthi2000conservation}.
\subsection{Riemann problem solutions}\label{sec:Riemann}
Continuing to focus on the two-flux setup as in section \ref{sec:2flux}, we consider Riemann data defined as follows
\begin{equation}
u_0(x):=\left\{\begin{array}{rl}
u_l&\mbox{ for }x<0,\\
u_r&\mbox{ for }x>0.
\end{array}\right.
\end{equation}
Below, we propose solutions $u$  for possible choices of $(u_l,u_r)$. In the next subsection we prove that these solutions satisfy the adapted entropy condition, hence they are unique. The notations $\theta_{+},\theta_-$ are as in section \ref{sec:2flux}. 

\medskip
\noindent{\bf Solution of the Riemann problem:}
\begin{enumerate}
\item $u_l \leq \theta_g^{+}$ and $u_r \geq \theta_f^{-}$. 
The solution is given by
\begin{equation*}
u(x,t)=\left\{\begin{array}{cc}
   u_l
   &   x<s_1t ,\\
 \displaystyle{\left(g^{\p}|_{(-\f,\theta_g^{-}]}\right)^{-1}\left(\frac{x}{t}\right)}
   &   s_1t\leq x\leq 0, \\
 \displaystyle{\left(f^{\p}|_{[\theta_f^{+},\f)}\right)^{-1}\left(\frac{x}{t}\right)}
   &   0\leq x\leq s_2t, \\
   u_r & \text{otherwise,}
  \end{array}  \right.
  \end{equation*}
  where $s_1=g'(u_l)$ and $s_2= f'(u_r)$.
  \item $u_l \leq \theta_g^{+}$ and $u_r \leq \theta_f^{+}$. Define $u_-=\left(g|_{(-\f,\theta_g^{-}]}\right)^{-1}(f(u_r))$.
 \begin{enumerate}
 	\item  Suppose $u_l < u_-$, then the solution is given by,
 	\begin{equation*}
 	u(x,t)=\left\{\begin{array}{cc}
 	u_l &   x<s_1t ,\\
 	\displaystyle{\left(g^{\p}|_{(-\f,\theta_g^{-}]}\right)^{-1}\left(\frac{x}{t}\right)}
 	&   s_1t\leq x\leq s_2t, \\
 	u_- & s_2t<x<0,\\
 	u_r & \text{otherwise,}
 	\end{array}  \right.
 	\end{equation*}
 	where $s_1=g'(u_l)$ and $s_2= g'(u_-)$.\\
 	
 	\item Suppose $u_l > u_-$, then the solution is given by,
 	\begin{equation*}
 	u(x,t)=\left\{\begin{array}{cc}
 	u_l &   x<s_1t, \\
 	u_- & s_1t< x< 0,\\
 	u_r & \text{otherwise,}
 	\end{array}  \right.
 	\end{equation*}
 	where $s_1=\frac{g(u_l)-g(u_-)}{u_l-u_-}$.
 \end{enumerate}
   
  \item $u_l \geq \theta_g^{-}$ and $u_r \geq \theta_f^{-}$.
  Define $u_+=\left(f|_{[\theta_f^{+},\f)}\right)^{-1}(g(u_l))$.\\
  \begin{enumerate}
  	\item Suppose $u_r > u_+$, then the solution is given by,
  \begin{equation*}
u(x,t)=\left\{\begin{array}{cc}
u_l &   x<0 ,\\
   u_+ &   0<x<s_1t, \\
 \displaystyle{\left(f^{\p}|_{[\theta_f^{+},\f)}\right)^{-1}\left(\frac{x}{t}\right)}
   &   s_1t\leq x\leq s_2t, \\
 u_r & x>s_2t,
   \end{array}  \right.
  \end{equation*}
  where $s_1=f'(u_+)$ and $s_2= f'(u_r)$.
  
   \item Suppose $u_r < u_+$, then the solution is given by,
  \begin{equation*}
u(x,t)=\left\{\begin{array}{cc}
   u_l &  x< 0 ,\\
   u_+ & 0<x<s_1t,\\
   u_r & \text{otherwise,}
  \end{array}  \right.
  \end{equation*}
  where $s_1=\frac{f(u_r)-f(u_+)}{u_r-u_+}.$
  
\end{enumerate}
  \item $u_l \geq \theta_g^{-}$ and $u_r \leq \theta_f^{+}$.
     \begin{enumerate}
   	\item Suppose $g(u_l)\geq f(u_r)$. Define $u_+=\left(f|_{[\theta_f^{+},\f)}\right)^{-1}(g(u_l))$, then the solution is given by,
    \begin{equation*}
u(x,t)=\left\{\begin{array}{cc}
   u_l &   x<0 ,\\
   u_+& 0<x<s_1t,\\
   u_r & \text{otherwise,}
  \end{array}  \right.
  \end{equation*}
  where $s_1=\frac{f(u_r)-f(u_+)}{u_r-u_+}$. 
  \item Suppose $g(u_l)\geq f(u_r)$. Define $u_-=\left(g|_{(-\f,\theta_g^{-}]}\right)^{-1}(f(u_r))$, then the solution is given by,
  \begin{equation*}
u(x,t)=\left\{\begin{array}{cc}
   u_l &   x<s_1t, \\
   u_- & s_1t<x<0,\\
   u_r & \text{otherwise,}
  \end{array}  \right.
  \end{equation*} 
  where $s_1=\frac{f(u_r)-f(u_-)}{u_r-u_-}.$ 
\end{enumerate}
\end{enumerate}

\begin{lemma}\label{lemma:Riemann}
	Let $A(x,u)=f(u)H(x)+g(u)(1-H(x))$ satisfy assumptions \eqref{A1}--\eqref{A3} and \eqref{A4}--\eqref{A5}. Let $u_l,u_r\in\R$ and $u_0:\R\rr\R$ be defined as 
	\begin{equation}
	u_0(x):=\left\{\begin{array}{rl}
	u_l&\mbox{ for }x<0,\\
	u_r&\mbox{ for }x>0.
	\end{array}\right.
	\end{equation}
	Let $u$ be defined by the above four-case solution of the Riemann problem. Then, $u$ is a weak solution to \eqref{eq:discont} and satisfies the adapted entropy condition in the sense of Definition \ref{def_adapted_entropy}.
\end{lemma}
 \begin{proof}
 	Follows from case by case check of the inequality $I_{(\theta_g,\theta_f)}(u_-,u_+)\geq0$ and an application of Lemma \ref{io}
 \end{proof}

 \subsection{Singular Maps} 
 
 \begin{definition}
 	Given a function $A(\cdot,\cdot)$ satisfying our assumptions \eqref{A1}--\eqref{A3} and \eqref{A4}-\eqref{A5}, we define corresponding singular maps by
 	\begin{equation}\label{def:Psi}
 	\Psi_A(x,u) := \int\limits_{u_M(x)}^u \left|\frac{\partial}{\partial u} A(x,s) ds\right|\mbox{ where }u_M(x)=\frac{u^-_M(x)+u^+_M(x)}{2},
 	\end{equation}
 \end{definition}
 which on simplification yields, %{\bf From John: Should we add a case: $\Psi_A(x,u)=0$ for $u \in (u_M^-(x),u_M^+(x)]$?}
 \begin{equation}
 \Psi_A(x,u)=
 \left\{\begin{array}{cl}
 \displaystyle -A(x,u), & \quad \text{if } u \leq u_M^-(x),\\
 0&\quad \text{if }u_M^-(x)< u\leq u_M^+(x),\\ 
 \displaystyle A(x,u),  &   \quad \text{if } u > u_M^+(x). 
 \end{array}  \right. 
 \end{equation}
 Since for each $x\in \R,$ $u \mapsto A(x,u)$ is convex, we have $u \mapsto \Psi_A(x,u)$ is non-decreasing. From the list of Riemann problem solution mentioned before, it is clear that in general total variation of the solution may exceed the total variation of the initial data, i.e. $|u_l-u_r|.$ %{\bf From John: Can we clarify the previous sentence?} 
 However for we do have such result for the transformation of the solution. This observation helps in proving the TV bounds on the transformations of the solution. We conclude this discussion in the form of the following lemma.
 
 \begin{lemma}\label{lemma:tvd}
 	Let $A(x,u)=H(-x)g(u)+H(x)f(u),$ where $f$ and $g$ are convex functions having same minimum. If $u$ is the adapted entropy solution of the IVP \eqref{eq:discont}--\eqref{eq:data}, for the Riemann data $u_0$ then, for $t \geq 0$
 	\begin{eqnarray}
 	TV\left(\Psi_A(\cdot,u(\cdot,t))\right)=  TV\left(\Psi_A(\cdot,u_0(\cdot))\right).
 	\end{eqnarray}
 \end{lemma}
 
 \begin{proof}
 	Proof can be verified case by case from the list of Riemann problem mentioned above.
 \end{proof}
  Note that invertibilty of $\Psi$ fails on $[u_M^-(x),u_M^+(x)]$. For this we introduce a new map $\pi:\R\times\R\rr\R$ defined as follows
 \begin{equation}
 \pi_A(x,u):=\left\{\begin{array}{cl}
 u_M^{-}(x)&\mbox{ if }u\leq u_M^{-}(x),\\
 u&\mbox{ if }u_M^-(x)\leq u\leq u_M^{+}(x),\\
 u_M^+(x)&\mbox{ if }u\geq u_M^{-}(x).
 \end{array}\right.
 \end{equation}
 Note that $u\mapsto \Psi(x,u)+\pi(x,u)$ is strictly increasing function. Therefore, it is invertible for each $x\in\R$. Now we have the following lemma,
 \begin{lemma}\label{lemma:TVD:pi}
 	Let $A(x,u)=H(-x)g(u)+H(x)f(u),$ where $f$ and $g$ are convex functions having same minimum. If $u$ is the adapted entropy solution of the IVP \eqref{eq:discont}--\eqref{eq:data}, for the Riemann data $u_0$ then, for $t \geq 0$
 	\begin{eqnarray}
 	TV\left(\pi_A(\cdot,u(\cdot,t))\right)=  TV\left(\pi_A(\cdot,u_0(\cdot))\right).
 	\end{eqnarray}%The map $t\mapsto TV(\pi^\de(\cdot,u^\de(\cdot,t)))$ decays for $t>0$.
 \end{lemma}
\subsection{Wave front tracking approximation}
In this section, we prove that there exists a solution to  IVP \eqref{eq:discont}--\eqref{eq:data}, extending the classical front tracking 
algorithm for a time independent flux to the case of a flux with spatial discontinuities satisfying the assumptions \eqref{A1}-\eqref{A3} and \eqref{A4}--\eqref{A5}. In other words, approximate solutions will be constructed by solving Riemann problems for an approximated equation, and then we pass to the limit in the sequence of approximate solutions to obtain the the solution of \eqref{eq:discont}--\eqref{eq:data}.  We begin this section by introducing the following definition.
\begin{definition}
	A set $\tilde{S}$ is said to be the completion of the set $S=\{u_1,\cdots,u_n\} \subset[-M,M],$ with respect to a function $A:\R \times \R \rightarrow \R$ if $\tilde{S}$ is the smallest set satisfying the following,
	$$\{u: A(x,u)=A(x,u_i), \text{ for some }x\in \R, u_i \in S\} \subset \tilde{S}.$$
	A set is said to be complete if $S=\tilde{S}.$
\end{definition}
The following lemma is an easy consequence of the above definition.
\begin{lemma}\label{lemma:completion}
	If $A$ has finitely many spatial discontinuities and satisfies the assumptions \eqref{A1} to \eqref{A3}, then given a finite set $S,$ its completion with respect to $A$ 
	is a finite set.
\end{lemma}
 
 Let $F:[a,b]\rr\R$ be a Lipschitz function and $S=\{s_i;1\leq j\leq m\mbox{ with }a=s_1<s_2<\cdots<s_m=b\}$ be a finite set of points. %Then we can approximate $F$ by piece-wise linear function $\tilde{F}$ constructed 
 We connect $(s_i,F(s_i))$ and $(s_{i+1},F(s_{i+1}))$ by line segments for $1\leq i\leq m-1$ and obtain a piece-wise linear function $\tilde{F}$. Henceforth we refer points of set $S$ as the break-points of the function $\tilde{F}$. % The set $S$ is% For a piecewise linear function $F$, we say 
 
 \noi\textbf{Front Tracking algorithm for discontinuous flux:}%\label{front_tracking}
 %\begin{figure}[h]
% 	\centering
% 	%        \includegraphics[width=.5\textwidth]{front_tracking_construction.eps}
% 	\caption{Front tracking Approximation}
% 	\label{Fig 1}
% \end{figure}
% 
% In this section, we prove that there exists a solution to  IVP \eqref{1}-\eqref{2}, extending classical front tracking for time independent flux to fluxes with spatial discontinuities satisfying the assumptions \eqref{A1}-\eqref{A3} and \eqref{A4}--\eqref{A5}. In other words, the approximated solutions will be constructed by solving Riemann problems for an approximated equation, and then we pass to the limit in the sequence of approximated solutions to obtain the the solution of \eqref{1}-\eqref{2}.  We begin this section by introducing the following definition.
% \begin{definition}
% 	A set $\tilde{S}$ is said to be completion of the set $S=\{u_1,\cdots,u_n\} \subset[-M,M],$ with respect to a function $A:\R \times \R \rightarrow \R$ if $\tilde{S}$ is the smallest set satisfying the following,
% 	$$\{u: A(x,u)=A(x,u_i), \text{ for some }x\in \R, u_i \in S\} \subset \tilde{S}.$$
% 	A set is said to be complete if $S=\tilde{S}.$
% \end{definition}
% The following lemma is an easy consequence of the above definition.
% \begin{lemma}\label{FT}
% 	If $A$ has finitely many spatial discontinuities and satisfies the assumptions \eqref{A1} to \eqref{A3}, then given a finite set $S,$ its completion with respect to $A$ 
% 	is a finite set.
% \end{lemma}
% 
% \noindent\textbf{Key steps to the proof:} 
% 
 \begin{enumerate}[label=\Roman*.]
 	\item Approximate the initial data $u_0$ by a piecewise constant function $u^{\delta}_0$ with finitely many discontinuities such that $u_0^\de\rr u_0$ in $L^1_{loc}(\R)$ as $\de\rr0$ and 
 	$$TV(\Psi_A(\cdot,u^{\de}_0))+TV(\pi_A(\cdot,u^\de_0))\leq C,$$%TV(\Psi_A(\cdot,u_0))+TV(\pi_A(\cdot,u_0)).$$
 	for some $C>0$ depends only on $TV(u_0)$, $TV(u^\pm_M)$ and the flux $A$.
 	\item \label{FT2}Approximate $A(x,u)$ by $A^{\delta}(x,u)$ such that $x\mapsto A^{\delta}(x,\cdot)$ is piecewise constant with finitely many discontinuities  and for fixed $x$
 	the function $u \rightarrow A^{\delta}(x,u)$ is piecewise linear convex and the set of break points  $P^{\delta}$ contains the range set of $u^{\delta}_0$ and is complete with respect to $A^{\delta}$. Then, $A^\de(x,u)\rr A(x,u)$ for a.e. $x\in\R$ and $\abs{u}\leq M$ for some $M>0$. Also in addition, $k_{\alpha}^{\delta, \pm }(x) \rightarrow k_{\alpha}^{\pm}(x)$ for a.e. $x\in\R$. 
 	\item Solve the following initial value problem via the wave front tracking algorithm,
 	\begin{eqnarray}
 	\label{eqnap1} u_t+\left(A^{\delta}(x,u)\right)_x&=& 0 \quad\quad\quad \text{for}\,\,\,(x,t) \in \mathbb{\R}\times (0,\infty),\\
 	\label{eqnap2} u(x,0)&=& u^{\delta}_0(x)\,\quad\text{for}\,\,\, x\in \mathbb{R}.
 	\end{eqnarray}
 	More precisely, at each discontinuity point of $u^\de_0$ we use the Riemann solver and obtain a solution $u^\de$ for small time. Discontinuities in the solution $u^\de$ are called fronts. Often we refer a front by $(u_L,u_R)$ if $u_L$, $u_R$ are left and right states along a discontinuity line of $u^\de$. Now after the first interaction happens at time $=t_0$, we again use the Riemann solver at each discontinuity point of $u^\de(\cdot,t_0)$ and proceed. Then we show that the number of fronts remains finite and we can continue the process for all time $t>0$ to obtain an approximate solution $u^{\de}$. Then, we show that $u^\de$ satisfies the adapted entropy condition \eqref{ineq:adapted}.
 	
 	\item As $A^{\delta}(x,u) \rightarrow A(x,u)$ and  $u^{\delta}_0 \rightarrow u_0$, show that the approximate solution $u^{\delta} \rightarrow u$, where $u$ is the adapted entropy solution to \eqref{eq:discont}--\eqref{eq:data}. 
 \end{enumerate}
 
 In what follows we denote the approximate solution as $u^\de$ which corresponds to flux $A^\de$ and initial data $u_0^\de$. We use the notation $\Psi^\de$ for the singular mapping $\Psi_{A^\de}$ corresponding to $A^\de$. For consistency of notation we use $\pi^\de$ in place of $\pi_{A^\de}$. Before we proceed further we have the following remark
 \begin{remark}\label{remark:scalar}
 Suppose $n$ many fronts $\{(u_{j},u_{j+1}),1\leq j\leq n\}$ meet at some $(x_0,t_0)$ which is away from interfaces. Since $u\mapsto A^\de(x,u)$ is convex for each $x\in\R$, we have%at any interaction away from interface we have 
 \begin{enumerate}[label=R-\arabic*]
 	\item\label{R-1} after interaction only one front $(u_1,u_{n+1})$ emanates. Hence \# of fronts decays,
 	\item\label{R-2} $\abs{u_1-u_{n+1}}\leq \sum\limits_{k=1}^{n}\abs{u_{k}-u_{k+1}}$.
 	\item Since $u\mapsto \Psi^\de(x,u)$ and $u\mapsto \pi(x,u)$ are non-decreasing functions, we have 
 	\begin{align*}
 	&\abs{\Psi^\de(x,u_1)-\Psi^\de(x,u_{n+1})}+\abs{\pi^\de(x,u_1)-\pi^\de(x,u_{n+1})}\\
 	&\leq \sum\limits_{k=1}^n\abs{\Psi^\de(x,u_k)-\Psi^\de(x,u_{k+1})}+\sum\limits_{k=1}^n\abs{\pi^\de(x,u_k)-\pi^\de(x,u_{k+1})}.
 	\end{align*}
 \end{enumerate}
For the proof of \eqref{R-1} and \eqref{R-2} we refer to \cite{Bressan}.  
 \end{remark}
% Now it will be shown that the functions satisfying the above condition indeed forms a large family of subsets of $L^{\infty},$ which include the space of functions of bounded variation.
In the next two lemmas we prove that $TV(\Psi_A(\cdot,u_0(\cdot))),TV(\pi_A(\cdot,u_0(\cdot)))<\f$ if $u_0\in BV(\R)$ and the flux $A$ satisfies \eqref{A1}--\eqref{A3} and \eqref{A4}--\eqref{A5}.
\begin{lemma}\label{lemma:Psi}
	Let $u \in BV$ and $\Psi_A$ be defined as in \eqref{def:Psi} for a flux $A$ satisfying \eqref{A1}--\eqref{A3} and \eqref{A4}--\eqref{A5}. Then the map $x\mapsto \Psi_A(x,u(x))$ is BV.% and $u\in BV$ for  then the following inclusion holds,
	% 		\begin{eqnarray}
	% 		\{u: TV\left(u\right) < \infty\}     \subset\{u: TV\left(\Psi(\cdot,u(\cdot))\right) < \infty\}.   
	% 		\end{eqnarray}
\end{lemma}

\begin{proof} Since $TV(u(\cdot))< \infty$, we have $M_0:=||u||_{L^\infty} < \infty$. 
	Thus, we have the following,
	\begin{eqnarray*}
		\left|\Psi_A(x,u)-\Psi_A(y,v)\right|&\leq& |\Psi_A(x,u)-\Psi_A(x,v)|+|\Psi_A(x,v)-\Psi_A(y,v)|\\
		%&\leq& K_1|u-v| +K_2|u_M(x)-u_M(y)|\\
		&\leq & K\left( |u-v| +|u_M(x)-u_M(y)|+\abs{a(x)-a(y)}\right).
	\end{eqnarray*}
	Hence, we get,
	\begin{eqnarray*}
		TV\left(\Psi(\cdot,u(\cdot)) \right)&=&  \sup \sum\limits_{i\in \mathbb{Z}} |\Psi(x_i,u(x_i))-\Psi(x_{i-1},u(x_{i-1}))|\\
		%&\leq & K\left( \sup \sum\limits_{i\in \mathbb{Z}} |u(x_i)-u(x_{i-1})|+\sup \sum\limits_{i\in \mathbb{Z}}|u_M(x_i)-u_M(x_{i-1})|\right)\\
		&\leq & K \left(TV (u) + TV (u_M)+TV(a)\right)<\f.
	\end{eqnarray*}
\end{proof}
 %satisfying $$TV(\Psi(\cdot,u_0(\cdot)))+TV(\pi(\cdot,u_0(\cdot)))<\f.$$ Then we use the stability estimate \eqref{estimate:stability} and the following lemma to conclude the result for BV data.
 \begin{lemma}\label{lemma:app:pi}
 	Let $w\in BV(\R)$ and $A$ be a flux satisfying \eqref{A1}--\eqref{A3} and \eqref{A4}--\eqref{A5}. Then the map $x\mapsto \pi_A(x,w(x))$ belongs to $BV(\R)$.% Let $w\in L^\f(\R)$ such that $TV(w)<\f$. Then there is a sequence of functions $w_n$ such that (i) $TV(\pi(w_n))<\f$ for each $n\geq1$, and (ii) $w_n\rr w$ in $L^1_{loc}(\R)$.
 \end{lemma}
\begin{proof}
	Let $x_1,x_2\in\R$. Then we have
	\begin{align}
	\abs{\pi_A(x_1,w(x_1))-\pi_A(x_2,w(x_2))}&\leq \abs{\pi_A(x_1,w(x_1))-\pi_A(x_1,w(x_2))}\nonumber\\
	&+\abs{\pi_A(x_1,w(x_2))-\pi_A(x_2,w(x_2))}.\label{inequality:pi1}
	\end{align}
	From the definition of $\pi_A$, $\abs{\pi_A(x,u_1)-\pi_A(x,u_2)}\leq \abs{u_1-u_2}$ for $x\in\R$ and $u_1,u_2\in\R$. We note that $u_M^-(x)\leq \pi_A(x,u)\leq u_M^+(x)$ and $u_M^-(y)\leq \pi_A(y,u)\leq u_M^+(y)$ for $x,y,u\in\R$. Subsequently, 
	\begin{equation}\label{ineq:pi2}
	\abs{\pi_A(x,u)-\pi_A(y,u)}\leq \abs{u_M^-(x)-u_M^-(y)}+\abs{u_M^+(x)-u_M^+(y)}\mbox{ for }x,y\in\R\mbox{ and }u\in\R.
	\end{equation}
	From \eqref{inequality:pi1} we have
    $$
	\abs{\pi_A(x_1,w(x_1))-\pi_A(x_2,w(x_2))}\leq \abs{w(x_1)-w(x_2)}+\abs{u_M^-(x)-u_M^-(y)}+\abs{u_M^+(x)-u_M^+(y)}.
	$$
  From \eqref{A5} we have $u^+_M(\cdot),u^-_M(\cdot)\in BV(\R)$. This completes the proof of Lemma \ref{lemma:app:pi}.
%	Since $w_n\in BV(\R)$, we can get an approximation sequence $\{w_n\}$ such that $w_n$ is piece-wise constant function with finitely many pieces and $w_n\rr w$ in $L^1_{loc}(\R)$ (see for instance \cite{Bressan}). Note that for a piecewise constant function $p(x)$ with finitely many pieces $\pi(p)<\f$. Hence $\pi(w_n)<\f$. This completes the proof of Lemma \ref{lemma:app:pi}.
\end{proof}

 We first prove the convergence of front-tracking approximation for BV data $u_0$. By Lemma \ref{lemma:Psi} and \eqref{lemma:app:pi} we obtain
 \begin{equation}
 TV(\Psi_A(\cdot,u_0(\cdot)))+TV(\pi_A(\cdot,u_0(\cdot)))<\f.
 \end{equation}
From the definition of $\Psi_A$ we can show that
\begin{equation}
\abs{\Psi_A(x,u)-\Psi_A(y,u)}\leq C[\abs{u_M(x)-u_M(y)}+\abs{a(x)-a(y)}]\mbox{ for }x,y\in\R,u\in[m,M]
\end{equation}
where $C$ depends only on $m,M$. Hence we have $TV(\Psi^\de(\cdot,u^\de_0))\leq TV(\Psi_A(\cdot,u^\de_0))+ C[TV(u_M)+TV(a)]$. 

%From the definition of $\pi_A$ we observe that 
%$$\abs{\pi_A(x,u)-\pi_A(y,u)}\leq \abs{u_M(x)-u_M(y)}+\abs{u_M^+(x)-u_M^-(x)}+\abs{u_M^+(y)-u_M^-(y)}.$$
By using \eqref{A5} and \eqref{ineq:pi2} we obtain
\begin{equation}
TV(\pi^\de(\cdot,u_0^\de(\cdot)))\leq TV(\pi_A(\cdot,u_0^\de(\cdot)))+TV(u_M(\cdot))+TV\left(u_M^+(\cdot)-u_M^-(\cdot)\right).
\end{equation}
Hence, 
\begin{equation}
 TV(\Psi_A^\de(\cdot,u_0^\de(\cdot)))+TV(\pi_A^\de(\cdot,u_0^\de(\cdot)))\leq C_1,
\end{equation}
for some $C_1>0$ does not depend on $\de$.
 
 We wish to keep track of fronts coming from the affine part of $A(x,u)$. %This deserves a special treatment as rarefaction can occur at interface even for $t>0$. 
 In this regards we would like to mention that recently in \cite{AG} it has been proved that rarefactions cannot occur if the flux does not have any affine part. Since we are considering more general fluxes which may contain affine parts, rarefactions can occur from interactions at the interfaces. That is why it deserves a special treatment in counting the number of fronts. For that purpose we define $Q(x)$ and $\# PL$ as follows.
 
  Let $Q(x)$ be the set of all pairs $(a,b)\in\R^2$ such that $A(x,u)$ is affine on $[a,b]$ and if $A(x,u)$ is affine on $[c,d]$ containing $[a,b]$, then $a=c,b=d$. Similarly we define the set $Q^\de(x)$ as the set of all pairs $(a,b)\in\R^2$ such that the following holds (i) $a,b\in P^\de$ with $(a,b)\cap P^\de\neq\emptyset$, (ii) the point $(v,A(x,v))$ lies on the straight line joining $(a,A(x,a))$ and $(b,A(x,b))$ for all $u\in (a,b)\cap P^\de$ and (iii) if $(c,d)\in Q^\de$ such that $[a,b]\subset [c,d]$ then $a=c$ and $b=d$.  %{\color{blue}Let $Q^\de(x)$ be such set for approximated flux $A^\de$.} 
  Note that by our choice of $A^\de$, we have $A^\de(x,u)=\sum\limits_{j=0}^{k}A(x_j,u)\chi_{(x_j,x_{j+1})}$ for $u\in P^\de$ where $-\f\leq x_0<x_1<\cdots<x_k<x_{k+1}\leq\f$. Then, we define
 \begin{equation}\label{def:PL}
 \#PL=\sum\limits_{j=0}^{k}\sum\limits_{(a,b)\in Q(x_j)}\#((a,b)\cap P^\de).
 \end{equation}
 Observe that $\#((a,b)\cap P^\de)=\#([a,b]\cap P^\de)-2$. The quantity $\# PL$ defined above depends on $\de$. For the sake of simplicity we suppress the parameter $\de$ in the notation. We also use the notation $\#PL([r_1,r_2],x)$ for the number $\#((r_1,r_2)\cap P^\de)$ where the pair $(r_1,r_2)\in Q^\de(x)$. The following lemma characterizes shocks under adapted entropy solution. %Proof runs on the similar lines of the proof in Godlevsky Raviart.This result will be used in the proof of \eqref{FT}
% {\bf From John: Should we explain the term break point before this lemma? Or possibly move this lemma so that it comes after 	the description of the FT algorithm.}
 \begin{lemma}\label{lemma:adapted-discrete}
 	For the flux $A(\cdot,\cdot)$ with finitely many spatial discontinuities (piecewise constant on spatial variable) and satisfying the assumptions \eqref{A1}--\eqref{A3} and \eqref{A4}--\eqref{A5}, a piecewise constant weak solution of \eqref{eq:discont}--\eqref{eq:data} satisfies the adapted entropy condition if and only if  the following holds:
 	\begin{enumerate}[label=S-\arabic*.]
 		\item $A(a-,u(a-,t))=A(a+,u(a+,t))$ for a discontinuity point $(a,t)\in\R\times(0,\f)$.
 		\item\label{one} Any discontinuity away form the interface of discontinuity separating the states  $(u_L,u_R)$ satisfies one of the following % satisfies  
 		\begin{enumerate}
 			\item $u_L>u_R$,
 			\item there exists no break points in $(u_R,u_L) $,
 			\item if there is a break-point in  $(u_R,u_L) $, then $u\mapsto A(x,u)$ is affine on $(u_R,u_L) $.
 		\end{enumerate}%either satisfies  or 
 		\item\label{two}[Interface discontinuity:] The discontinuity at the interface $x=a$ separating the states  $(u_L,u_R)$ satisfies one of the following%If $(u_L,u_R)$ is a discontinuity along the spatial discontinuity $x=a$ of $A(x,u)$, then one of the following holds,
 		\begin{enumerate} 
 			\item $u_L\leq u_M^{+}(a_-)$ and $u_R\leq u_M^{+}(a_+)$, 
 			\item $u_L\geq u_M^{-}(a_-)$ and $u_R\geq u_M^{-}(a_+)$,
 			\item $u_L\geq u_M^{-}(a_-)$ and $u_R\leq u_M^{+}(a_+).$
 		\end{enumerate}
 	\end{enumerate}
 \end{lemma}
 Before we proceed to the proof of Lemma \ref{lemma:adapted-discrete}, we  remark that away from interface discontinuities of type (a), (b), (c) of \eqref{one} will often be referred later in this article by shock, rarefaction and contact discontinuity respectively.
 \begin{proof}[Proof of Lemma \ref{lemma:adapted-discrete}:]
 	If shocks are away from the interface, the adapted entropy condition boils down to the Kruzkhov entropy condition and (\ref{one}) follows as in the case of conservation law with spatially independent flux. To prove (\ref{two}), suppose $u_L< u_M(a_-)$ and $u_R> u_M(a_+).$ Applying Rankine-Hugoniot condition along the line of discontinuity of $x\mapsto A(x,\cdot)$ we have,
 	$$(A(a_-,u_L)-A(a_-,u_M(a_-)))=(A(a_+,u_R)-A(a_+,u_M(a_+)))>0.$$
 	Substituting it in the definition of interface entropy functional, we get,
 	\begin{equation*}
 	I_{\left(u_M(a_-),u_M(a_+)\right)}=\left(\sgn(u_L-u_M(a_-))-\sgn(u_R-u_M(a_+))\right)(A(a_+,u_R)-A(a_+,u_M(a_+))) <0,
 	\end{equation*}
 	which is a contradiction to the Lemma \ref{io}, hence one of the possibilities (a), (b), (c) of \eqref{two} must hold.
 \end{proof}

 Next we define the minimum variation $\xi_\de$ for different states at $\de$-approximation level.
 \begin{equation}
 \xi_\de:=\min\limits_{0\leq j\leq k}\min\limits_{(u_1,u_2)\in P^\de\times P^\de\setminus([u_M^-(x_j),u_M^+(x_j)]\times[u_M^-(x_j),u_M^+(x_j)]),\\
 	\,u_1\neq u_2}\abs{\Psi^\de(x_j,u_1)-\Psi^\de(x_j,u_2)}.
 \end{equation}
 The quantity $\xi_\de$ helps us to show that up to finite time $T>0$ the decay of $TV (\Psi^\de(\cdot,u(\cdot,t)))$ can not happen infinitely many times which implies that the number of interactions where $TV (\Psi^\de(\cdot,u(\cdot,t)))$ decays can not happen infinitely many times. In the proof of the following lemma and in the rest of the article we sometimes use the notation $\Psi_g(v)=\Psi^\de(x_j-,v)$ and $\Psi_f(v)=\Psi^\de(x_j+,v)$. Now we have the following lemma which is key to prove finiteness of fronts.
\begin{lemma}\label{lemma:interface}
	Let $t_0$ be a time when one or many fronts hit an interface. Then the following two properties hold:
	\begin{enumerate}
		\item either $TV(\Psi^\de(\cdot,u^\de(\cdot,t_0+)))-TV(\Psi^\de(\cdot,u^\de(\cdot,t_0-)))\leq -\xi_\de$,
		\item or the number of fronts increase by at most $\# PL$ where $\# PL$ is defined as in \eqref{def:PL}.%$\sum\limits_{j=-m}^{m}\#PL^n(A(x_j,\cdot))$.
	\end{enumerate}
\end{lemma} %\noi\textbf{States of linear part:} {\color{blue}As we mentioned before, linearity on $f_-$}
  
%  \begin{lemma}
%  	Let $t_0$ be a time when one or many fronts hit an interface. Then the number of fronts increase by at most $\sum\limits_{j=-m}^{m}\#PL^n(A(x_j,\cdot))$.%the following two properties hold:
%%  	\begin{enumerate}
%%  		\item either $TV(\Psi(\cdot,u^\de(\cdot,t_0+)))-TV(\Psi(\cdot,u^\de(\cdot,t_0+)))<-\xi_\de$,
%%  		\item or the number of fronts increase by at most $\sum\limits_{j=-m}^{m}\#PL^n(A(x_j,\cdot))$.
%%  	\end{enumerate}
%  \end{lemma} %\no
  \begin{proof}
 Suppose at $t=t_0$, $l$ many fronts $(u_{m-l},u_{m-l+1}),\cdots,(u_{m-1},u_{m})$ meet the interface from left and $k$ many fronts $(u_{m+1},u_{m+2}),\cdots,(u_{m+k},u_{m+k+1})$ meet the interface $x=x_j$ from right. For our convenience we denote $f=A^\de(x_j+,\cdot)$, $g=A^\de(x_j-,\cdot)$ and $u_M^\pm(x_j-)=\theta_g^\pm,u_M^\pm(x_j+)=\theta_f^\pm$. Then, 
 %By using entropy condition \eqref{Oleinik-entropy} we have 
 \begin{enumerate}[(i)]
 	\item $f(u_{m+1})=g(u_{m})$.
 	
 	\item speeds of fronts are decreasing, in particular, we have
 	\begin{equation}
 	\frac{g(u_{m-l})-g(u_{m-l+1})}{u_{m-l}-u_{m-l+1}}\geq \frac{g(u_{m-1})-g(u_{m})}{u_{m-1}-u_{m}}\mbox{ and } \frac{f(u_{m+1})-f(u_{m+2})}{u_{m+1}-u_{m+2}}\geq \frac{f(u_{m+k})-f(u_{m+k+1})}{u_{m+k}-u_{m+k+1}}.
 	\end{equation}
 \end{enumerate}
 
 %From Lemma \ref{} we have .
 Next, we wish to show that in all possible types of interaction at the interface, either $TV$ of $\Psi$ strictly decreases or fronts at the interface are coming from affine part of flux. 
 \begin{enumerate}[label=O.\arabic*]
 	\item\label{I1} \textit{When $l=0$, $k\geq1$ and $u_{m}\leq\theta_g^{+}$:} In this case we observe that $u_{m+k+1}\leq\theta_f^{+}$. After interaction, we solve the Riemann problem $(u_m,u_{m+k+1})$ at the interface. Let $u_-=\left(g|_{(-\f,\theta_g^{-}]}\right)^{-1}(f(u_{m+k+1}))$. From section \ref{sec:Riemann}, we note that there are two possibilities: %By previous observation we know that there are two possibilities
 	\begin{enumerate}[label=O.1.\arabic*]
 		\item\label{I11} First we consider the case when $u_->u_m$. Note that in this case either rarefaction wave or contact discontinuity or mixture of them appears. Therefore, number of fronts may increase in this case. Since $u_->u_m$ and $u_-,u_m\leq \theta_g^{+}$ we have $f(u_{m+k+1})=g(u_-)<g(u_m)=f(u_{m+1})$. By Lemma \ref{lemma:adapted-discrete} we have $u_{m+1}\leq \theta_f^{+}$. Hence $\Psi^\de(x_j-,u_m)=\Psi^\de(x_j+,u_{m+1})$. Change in $TV(\Psi)$ due to this interaction at interface is
 		\begin{align}
 		&\abs{\Psi_g(u_m)-\Psi_g(u_-)}-\sum\limits_{n=1}^{k}\abs{\Psi_f(u_{m+n})-\Psi_f(u_{m+n+1})}\nonumber\\
 		=&\abs{\Psi_f(u_{m+1})-\Psi_f(u_{m+k+1})}-\sum\limits_{n=1}^{k}\abs{\Psi_f(u_{m+n})-\Psi_f(u_{m+n+1})}.\label{change:TVPsi}
 		\end{align}
 		We observe that \eqref{change:TVPsi} becomes $0$ if and only if $\Psi_f(u_{m+1})<\cdots<\Psi_f(u_{m+k+1})$.  Since $\Psi_f$ is increasing function we have $u_{m+1}<\cdots<u_{m+k+1}$. By Lemma \ref{lemma:adapted-discrete} we get two further possibilities: (i) either $k=1$ and there is no break-point between $u_{m+1}$ and $u_{m+2}$, (ii) or $u_{m+1}$ and $u_{m+k+1}$ are connected by linear segment. Now we observe that if (i) holds then number of fronts remains same after interaction.  On the other hand if $u_{m+1}$ and $u_{m+k+1}$ are connected by linear segment then they can not meet if $k\geq2$. Hence, $k=1$. In this case number of fronts increases by $\# PL([u_{m+1},u_{m+2}],x_j)-1$. (see figure \ref{figure:WFT-1})

 		\begin{figure}[ht]
 			\centering
 			\begin{tikzpicture}

 			\draw[thick,->] (-10 ,0) -- (0,0) node[anchor=north west] {$x$};
 			\draw[thick,->] (-10 ,0) -- (-10,5) node[anchor=north east] {$t$};
 			
 			\draw[line width=0.2mm, blue] (-7 ,0) -- (-7,4.5) ;
 			\draw[line width=0.2mm, blue] (-5.8 ,0) -- (-5.8,4.5) ;
 			\draw[line width=0.2mm, blue] (-5 ,0) -- (-5,4.5) ;
 			\draw[line width=0.2mm, blue] (-3.5 ,0) -- (-3.5,4.5) ;
 			\draw[line width=0.2mm, blue] (-1.5 ,0) -- (-1.5,4.5) ;
 			\draw[line width=0.2mm, black] (-4.8 ,0) -- (-3.5,1.5)-- (-1.5,3)-- (-1,4);
 			\draw[line width=0.2mm, black] (-1.5 ,3) -- (-.8,4);
 			\draw[line width=0.2mm, black] (-1.5 ,3) -- (-0.6,4);
 			\draw[line width=0.2mm, black] (-6.4 ,0) -- (-7,2)-- (-8,3)-- (-8.5,4);
 			\draw[line width=0.2mm, black] (-7 ,0) -- (-7.6,1.5)-- (-8,3);
 			\draw[line width=0.2mm, black] (-9 ,0) -- (-7.6,1.5);

 					\draw[thick][] (-1.5,3) node[anchor=east] {\tiny $(x_0,t_0)$};
 					\draw[thick][] (-3.5,1.5) node[anchor=east] {\tiny $(x_1,t_1)$};
 					\draw[thick][] (-7,2) node[anchor=west] {\tiny $(x_2,t_2)$};
 					
 					\draw[thick][] (-8,3) node[anchor=east] {\tiny $(x_3,t_3)$};
 					\draw[thick][] (-7.6,1.5) node[anchor=east] {\tiny $(x_4,t_4)$};
 					
 				\draw[thick][] (-4.3,.8) node[anchor=east] {\tiny $u_2$};
 					\draw[thick][] (-4.2,.5) node[anchor=west] {\tiny $u_1$};
 						\draw[thick][] (-2.7,2.4) node[anchor=east] {\tiny $v_2$};
 					\draw[thick][] (-2.4,2) node[anchor=west] {\tiny $v_1$};
 						\draw[thick][] (-1.6,3.9) node[anchor=west] {\tiny $w_2$};
 					\draw[thick][] (-1.2,2.5) node[anchor=west] {\tiny $w_1$};

 					\draw[thick][] (-5.8,.5) node[anchor=east] {\tiny $u_4$};
 					\draw[thick][] (-5.6,.8) node[anchor=west] {\tiny $u_3$};
 						\draw[thick][] (-6.5,.4) node[anchor=east] {\tiny $u_5$};
 							\draw[thick][] (-7.1,1.8) node[anchor=east] {\tiny $u_6$};
 								\draw[thick][] (-7.1,3) node[anchor=east] {\tiny $u_7$};
 								\draw[thick][] (-7.5,.6) node[anchor=east] {\tiny $u_8$};
 								\draw[thick][] (-8.8,1.8) node[anchor=east] {\tiny $u_9$};

 				\draw[thick][] (-7,4.5) node[anchor=south] {\tiny $x=x_5$};
 					\draw[thick][] (-5.8,4.5) node[anchor=south] {\tiny $x=x_4$};
 						\draw[thick][] (-4.7,4.5) node[anchor=south] {\tiny $x=x_3$};
 							\draw[thick][] (-3.5,4.5) node[anchor=south] {\tiny $x=x_2$};
 								\draw[thick][] (-1.5,4.5) node[anchor=south] {\tiny $x=x_1$};
 			\end{tikzpicture}
 			\caption{This picture illustrates the interaction of fronts (i) at interface and (ii) away from interfaces. In this figure, we consider 5 interfaces at $x=x_i,i=1,\cdots,5$.  Fronts $(u_1,u_2)$ and $(v_1,v_2)$ are contact discontinuities and at time $t=t_0$ rarefaction fronts are created from interface $x=x_0$. On the other hand $(x_3,t_3),\,(x_4,t_4)$ represent the interaction points which are away from interface. We notice that only one shock created after the interaction at $(x_3,t_3)$ and $(x_4,t_4)$. }\label{figure:WFT-1}
 		\end{figure}
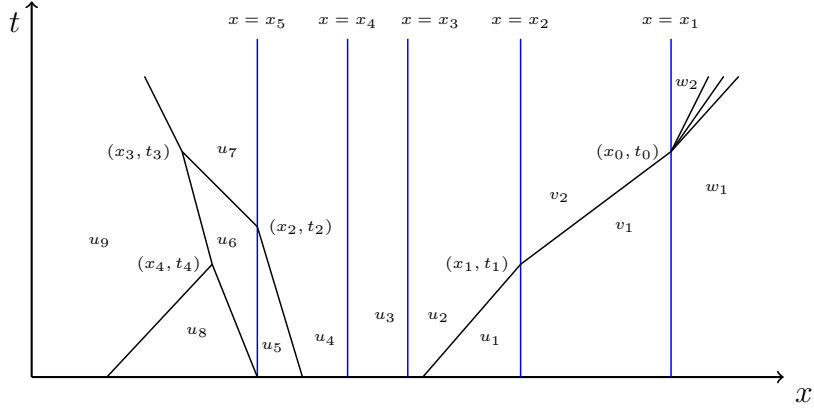

 		\item If  $u_-<u_m$ then shock appears on left side of the interface. Hence number of fronts does not increase.
 	\end{enumerate}
 	\item\label{I2}\textit{When $l=0$, $k\geq1$ and $u_{m}>\theta_g^{-}$:} We observe that $u_{m+k+1}<\theta_f^{+}$ since $(u_{m+k},u_{m+k+1})$ front meets the interface from right. Hence by the Riemann problem solutions as in section \ref{sec:Riemann} we know that only one shock appears either on right or left side of the interface. Hence number of fronts does not increase. (see figure \ref{figure:WFT-2})
 	\item\label{I3}\textit{When $l\geq1$, $k=0$ and $u_{m+1}\geq\theta_f^-$:} This case is similar to (\ref{I1}). Here we note that $u_{m-l}\leq \theta_g^+$. After interaction at interface, we solve the Riemann problem $(u_{m-l},u_{m+1})$ according to section \ref{sec:Riemann}. Let $u_+=\left(f|_{[\theta_f^{+},\f)}\right)^{-1}(g(u_{m-l}))$. By previous observation we know that there are two possibilities:
 	\begin{enumerate}[label=O.3.\arabic*]
 		\item\label{I31} First we consider the case when $u_+<u_{m+1}$. Note that in this case either rarefaction wave or contact discontinuity or mixture of them appears. Therefore, number of fronts may increase in this case. Since $u_+<u_{m+1}$ and $u_+,u_{m+1}\geq \theta_f^-$ we have $g(u_{m-l})=f(u_+)<f(u_{m+1})=g(u_{m})$. By Lemma \ref{lemma:adapted-discrete} we have $u_{m}\geq \theta_g^-$. Hence $\Psi^\de(x_j-,u_m)=\Psi^\de(x_j+,u_{m+1})$. Change in $TV(\Psi^\de)$ due to this interaction at interface is
 		\begin{align}
 		&\abs{\Psi_f(u_{m+1})-\Psi_f(u_+)}-\sum\limits_{n=1}^{l}\abs{\Psi_g(u_{m-n})-\Psi_g(u_{m-n+1})}\nonumber\\
 		=&\abs{\Psi_g(u_{m})-\Psi_g(u_{m-l})}-\sum\limits_{n=1}^{l}\abs{\Psi_g(u_{m-n})-\Psi_g(u_{m-n+1})}.\label{change:TVPsi1}
 		\end{align}
 		We observe that \eqref{change:TVPsi1} becomes $0$ if and only if $\Psi_g(u_{m-l})<\cdots<\Psi_g(u_{m})$.  Since $\Psi_g$ is increasing function we have $u_{m-l}<\cdots<u_{m}$. By Lemma \ref{lemma:adapted-discrete} we get two possibilities: (i) either $k=1$ and there is no break-point between $u_{m-l}$ and $u_{m}$, (ii) or $u_{m-l}$ and $u_{m}$ are connected by linear segment. Now we observe that if (i) holds then number of fronts remains same after interaction. Note that if $u_{m-l}$ and $u_{m}$ are connected by linear segment then they can not meet if $l\geq2$. Hence, $l=1$. In this case number of fronts increases by $\# PL([u_{m-1},u_{m}],x_j)-1$.

 				 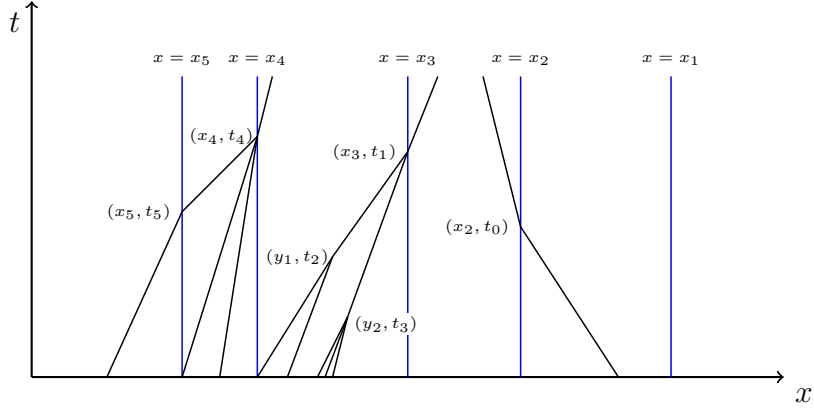
\begin{figure}[ht]
 			\centering
 			\begin{tikzpicture}

 			\draw[thick,->] (-10 ,0) -- (0,0) node[anchor=north west] {$x$};
 			\draw[thick,->] (-10 ,0) -- (-10,5) node[anchor=north east] {$t$};
 			
 				\draw[line width=0.2mm, blue] (-8 ,0) -- (-8,4) ;
 			\draw[line width=0.2mm, blue] (-7 ,0) -- (-7,4) ;
 		
 			\draw[line width=0.2mm, blue] (-5 ,0) -- (-5,4) ;
 			\draw[line width=0.2mm, blue] (-3.5 ,0) -- (-3.5,4) ;
 			\draw[line width=0.2mm, blue] (-1.5 ,0) -- (-1.5,4) ;
 			\draw[line width=0.2mm, black] (-7 ,0) -- (-6,1.6)-- (-5,3)-- (-4.6,4);
			\draw[line width=0.2mm, black] (-6 ,1.6) -- (-6.6,0);
 			\draw[line width=0.2mm, black] (-6.2,0) -- (-5.8,.8) -- (-5,3);
 			\draw[line width=0.2mm, black] (-5.8,0.8) -- (-6.1,0);
 			\draw[line width=0.2mm, black] (-5.8 ,0.8) -- (-6,0);
 			\draw[line width=0.2mm, black] (-2.2 ,0) -- (-3.5,2)-- (-4,4);
 			\draw[line width=0.2mm, black] (-9 ,0) -- (-8,2.2)-- (-7,3.2)--(-6.8,4);
 			\draw[line width=0.2mm, black] (-8 ,0) -- (-7,3.2);
 			\draw[line width=0.2mm, black] (-7.5 ,0) -- (-7,3.2);

 			\draw[thick][] (-8,4) node[anchor=south] {\tiny $x=x_5$};
 			\draw[thick][] (-7,4) node[anchor=south] {\tiny $x=x_4$};
 			\draw[thick][] (-5,4) node[anchor=south] {\tiny $x=x_3$};
 			\draw[thick][] (-3.5,4) node[anchor=south] {\tiny $x=x_2$};
 			\draw[thick][] (-1.5,4) node[anchor=south] {\tiny $x=x_1$};

 			\draw[thick][] (-3.5,2) node[anchor=east] {\tiny $(x_2,t_0)$};
 			\draw[thick][] (-5,3) node[anchor=east] {\tiny $(x_3,t_1)$};
 			 				\draw[thick][] (-6.9,3.2) node[anchor=east] {\tiny $(x_4,t_4)$};
 				\draw[thick][] (-8,2.2) node[anchor=east] {\tiny $(x_5,t_5)$};
 				\draw[thick][] (-5.9,1.6) node[anchor=east] {\tiny $(y_1,t_2)$};
 				\draw (-5.75,0.7) node[fill=white, inner sep=1pt, anchor=west] {\tiny $(y_2,t_3)$};
 			\end{tikzpicture}
 			\caption{This illustrates possible wave interactions for a front tracking approximation with piecewise constant flux with 5 spatial discontinuities at $x=x_i, \,i=1,\cdots,5$. At the interaction points $(x_1,t_1)$ and $(x_4,t_4)$ more than one fronts hit the interface and as a result only one front arises from interface after the interaction. Note that the interactions at $(y_1,t_2),(y_2,t_3)$ are away from interface and only one front emanates after interaction. }\label{figure:WFT-2}
 		\end{figure}

 		\item If  $u_+>u_m$ then shock appears on right side of the interface. Hence number of fronts does not increase.
 	\end{enumerate}
 	\item\label{I4}\textit{When $l\geq1$, $k=0$ and $u_{m+1}<\theta_f^+$:} We observe that $u_{m-l}>\theta_g^-$ since $(u_{m-l},u_{m-l+1})$ front meets the interface from left. Hence by Riemann problem solutions in section \ref{sec:Riemann}, we know that only one shock appears either on right or left side of the interface. Hence number of fronts does not increase.
 	
 	\item\label{I5}\textit{When $l\geq1$, $k\geq1$:} We observe that $u_{m-l}>\theta_g^-$ since $(u_{m-l},u_{m-l+1})$ front meets the interface. Similarly, we have $u_{m+k+1}<\theta_f^+$ since $(u_{m+k},u_{m+k+1})$ front meets the interface from right. Hence by Riemann problem solution as in section \ref{sec:Riemann} we know that only one shock appears either on right or left side of the interface (see figure \ref{figure:WFT-2}). Hence number of fronts does not increase.
 \end{enumerate}
   \end{proof}
 
 As it is observed in (\ref{I11}) and (\ref{I31}), number of fronts increases if a contact discontinuity hits the interface. We want to keep track of this increase of fronts. In the next lemma we prove that a contact wave can not be generated from an interaction away from interfaces. 
 \begin{lemma}\label{lemma:contact-dis}
 	Suppose fronts $(u_j, u_{j+1})$ for $1\leq j\leq k-1$, $k\geq3$ meet away from interface. Then either $u_1>u_k$ or there is no break-point in the interval $(u_1,u_k)$.
 \end{lemma}
 
 \begin{proof}
 	Note that fronts $(u_j, u_{j+1})$ for $1\leq j\leq k-1$ are generated by Riemann solver. Hence, there are three possibilities: (i) $u_j>u_{j+1}$, (ii) $u_{j+1}>u_j$ and there is no break-point in the interval $(u_j,u_{j+1})$, (iii) $u_{j+1}>u_j$ and the flux is linear on the interval $(u_j,u_{j+1})$. Suppose $u_1<u_k$ and there are break-point in $(u_1,u_k)$. Let $j_0$ be the smallest index so that $u_1<u_{j_0}$ and there is atleast one break-point in $(u_1,u_{j_0})$. Now there are two possibilities:
 	\begin{enumerate}[label=L-\arabic*]
 		\item\label{L1} We first consider the case when $j_0\geq3$. Then note that from minimality of $j_0$ we have either $u_{j_0-1}<u_1$ or there is no break-point in $(u_1,u_{j_0-1})$. Hence, 
 		\begin{enumerate}
 			\item if $u_{j_0-1}<u_1$, then we have $u_{j_0-1}<u_{j_0}$ since $u_{j_0}<u_1$.
 			
 			\item if $u_{j_0-1}>u_1$ and there is no break-point in $(u_1,u_{j_0-1})$ then $u_{j_0-1}<u_{j_0}$ since $u_{j_0}$ is also a break point and $u_{j_0}>u_1$.
 			
 		\end{enumerate}
 		Therefore, we have $u_{j_0-1}<u_{j_0}$. By a similar argument, we have $u_{j_0-2}<u_{j_0}$. %Since $(u_{j_0-1},u_{j_0})$ satisfies (i)--(iii), we have Similarly, we have either $u_{j_0-2}<u_1$ or there is no break-point in $(u_1,u_{j_0-2})$. 
 		Since $(u_{j_0-2},u_{j_0-1})$ and $(u_{j_0-1},u_{j_0})$ meet and $f$ is convex, we have $u_{j_0-2}>u_{j_0}$. This gives a contradiction. Hence, Lemma \ref{lemma:contact-dis} is proved for this case.
 		\item Now we consider the case when $j_0=2$. Suppose $j_1$ is the smallest number in $\{3,\cdots, k\}$ such that there is atleast one break-point in $(u_1,u_{j_1})$. By a similar argument as in (\ref{L1}) we show that $j_1$ can not satisfy $j_1>3$. Now we consider when $j_1=3$. Since $(u_1,u_2)$ meets with $(u_2,u_3)$, we obtain $u_1>u_3$. This gives a contradiction. Hence Lemma \ref{lemma:contact-dis} is proved.
 	\end{enumerate}
 \end{proof}
 
 Therefore, from Lemma \ref{lemma:interface} and \ref{lemma:contact-dis}, we note that new contact discontinuity can be created only at initial time $t=0+$. Hence, we observe that
 \begin{equation}
 \# [\mbox{ of fronts at any time } t]\leq  \#[\mbox{ of fronts at time } t=0+]\times \left[\#PL \right],
 \end{equation}
 where $\# PL$ is defined as in \eqref{def:PL}. By Remark \ref{remark:scalar}, Lemma \ref{lemma:app:pi}, \ref{lemma:adapted-discrete}, \ref{lemma:interface} and \ref{lemma:contact-dis} we obtain
 \begin{align}
 TV(\Psi^\de(\cdot,u^{\de}(\cdot,t)))+TV(\pi^\de(\cdot,u^\de(\cdot,t)))&\leq TV(\Psi^\de(\cdot,u^{\de}_0))+TV(\pi(\cdot,u^\de_0))\nonumber\\
 &\leq TV(\Psi(\cdot,u_0))+TV(\pi(\cdot,u_0))\nonumber\\
 &+C[TV(u_M)+TV(u_M^+-u_M^-)+TV(a)]\nonumber\\
 &<\f.\label{estimate:TVB-Psi-u}
 \end{align}%^n(A(x_j,\cdot))$ is defined as total number of breakpoints lying on affine parts of $A(x_j,\cdot)$ at the $n$-th approximation level.

  \begin{lemma}
 	If $u_0 \in L^{\infty},$ then there exists an $M$ such that $||u^{\delta}||_{L^{\infty}} \leq M.$
 \end{lemma}
 \begin{proof}
 	Let $u^{\delta}_0 \in [-K,K]$ for some $K >0.$ Define the following, 
 	\begin{eqnarray}
 	%K_1=\min\limits_{x\in \mathbb{R},u\in [-K,K]}A(x,u) \geq \min\limits_{u\in [-K,K]}g(u) > -\infty\\
 	\alpha:=\max\limits_{x\in \mathbb{R},u\in [-K,K]}A^{\delta}(x,u^{\delta}_0).
 	\end{eqnarray}
 	Thus for each of the front tracking approximations $u^{\delta},$ the following holds. 
 	\begin{eqnarray}
 	0\leq A^{\delta}(x,u_0^{\delta}(x,t)) \leq \alpha 
 	\end{eqnarray}
 	Due to the assumption \ref{A2} we have,
 	\begin{equation}
 	\alpha\leq \max\limits_{u\in [-K,K]}g(u_0^{\delta}) < \infty,
 	\end{equation}
 	which implies
 	\begin{eqnarray}g^{-1}_-(\alpha)\leq u^{\delta} \leq g^{-1}_+(\alpha).\end{eqnarray}
 	Taking $M=\max\{|g^{-1}_-(\alpha)|,|g^{-1}_+(\alpha)|\}$ completes proof.
 \end{proof}
\begin{lemma}\label{timecontinuity}
	 Let $v^\de$ be defined as $v^\de(x,t):=\Psi^\de(x,u^\de(x,t))+\pi^\de(x,u^\de(x,t))$. For $0\leq t_1<t_2,$ each of the front tracking approximation satisfies the following time continuity property.
	\begin{eqnarray}
	||v^{\delta}(\cdot,t_1)-v^{\delta}(\cdot,t_2)||_{L^1(\R)} \leq L|t_1-t_2|TV(v^{\delta}_0).
	\end{eqnarray}
\end{lemma}
\begin{proof}
	Since the approximate solution $u^\de$ satisfies the Rankine-Hugoniot condition at each discontinuity line, we have
	\begin{equation}
	\|u^\de(\cdot,t_1)-u^\de(\cdot,t_2)\|_{L^1(\R)}\leq TV(A^\de(\cdot,u^\de(\cdot,t_1))) |t_1-t_2| \mbox{ for }t_1,t_2\in [\underline{t},\overline{t}],
	\end{equation}
	where $\underline{t},\overline{t}$ are two points in $\R_+$ such that there is no wave interaction in $(\underline{t},\overline{t})$. Since 
	\[TV(A^\de(\cdot,u^\de(\cdot,t_1)))\leq TV(\Psi^\de(\cdot,u^\de(\cdot,t_1)))\leq TV(v_0^\de),\] we have
	\begin{equation}
		\|u^\de(\cdot,t_1)-u^\de(\cdot,t_2)\|_{L^1(\R)}\leq  TV(v_0^\de)|t_1-t_2|\mbox{ for }t_1,t_2>0,%\in [\underline{t},\overline{t}],
	\end{equation} 
	thus we get %We note that
	\begin{equation}
			\|v^\de(\cdot,t_1)-v^\de(\cdot,t_2)\|_{L^1(\R)}\leq L TV(v_0^\de)|t_1-t_2|\mbox{ for }t_1,t_2>0,%\in [\underline{t},\overline{t}],
	\end{equation}
	where $L$ is the uniform Lipschitz constant of $u\mapsto\Psi^\de(x,u)+\pi^\de(x,u)$. 
%	In the interior, time interaction follows from standard argument. So we give the estimates only when the interactions occur at the interface.
%	In each of the $I_i$ fluxes are independent of time and non classical shocks does not contribute to the total variation of $v^{\delta}$ thus repeating the arguments as in the case of continuous flux we get, 
%	\begin{eqnarray}||v^{\delta}(\cdot,t_2)-v^{\delta}(\cdot,t_1)||_{L^1}&\leq & |t_2-t_1|TV A^{\delta}(\cdot,u^{\delta}(\cdot,t_1))\\
%	&\leq & |t_2-t_1|TV (v^{\delta}(\cdot,t_1))\\
%	&\leq & |t_2-t_1|TV (v_0^{\delta}).
%	\end{eqnarray}
\end{proof}

 	The following lemma shows that there exists a sequence of approximations of $u_0$ and $A,$ such that 
	part \eqref{FT2} of the front tracking algorithm holds. %This proof is constructive and motivated by the ideas in \cite{bressan2000hyperbolic,piccoli2018general}. 
 	\begin{lemma} Let $u_0:\R\rr\R$ be a function satisfying (i) $m\leq u_0(x)\leq M$ for some $0<m<M<\f$ and (ii) $TV(u_0)<\f$. Let $A:\R \times \R \rightarrow \R$ satisfy the hypothesis \eqref{A1}-\eqref{A3} and \eqref{A4}-\eqref{A5}. Then there exist approximate sequences $\{u_0^n\}$ and $\{A^n\}$ such that
 		\begin{enumerate}[(a)]
 			\item\label{a} $u_0^n$ is piecewise constant with finitely many discontinuities and $m\leq u_0^n(x)\leq M$ for a.e. $x\in\R,\,n\geq1$. 
 			\item\label{b} $u_0^n\rr u_0$ in $L^1_{loc}(\R)$ as $n\rr\f$.
 			\item\label{c} $TV(\Psi_A(\cdot,u^n_0))+TV(\pi_A(\cdot,u^n_0))\leq C$ for $n\geq1$ where $C>0$ depends only on $TV(u_0)$, $TV(u_M^\pm)$ and the flux $A$.
 			\item $A^n:\R \times \R \rightarrow \R$ has finitely many spatial discontinuities and $A^n \rightarrow A$ pointwise in $\R \times \R$.
 			\item For fixed $x$	the function $u \rightarrow A^{n}(x,u)$ is piecewise linear convex and the set of break points  $P^{n}$ contains the range set of $u^{n}_0$ and  complete with respect to $A^{n}$. 
 				\end{enumerate}
% 	 of Approximation $A^n:\R \times \R \rightarrow \R$ such that $A^n \rightarrow A$ pointwise in $\R \times \R$, $A^{n}$ has finitely many spatial discontinuities and for fixed $x$
% 		the function $u \rightarrow A^{n}(x,u)$ is piecewise linear convex and the set of break points  $P^{n}$ contains the range set of $u^{n}_0$ and  complete with respect to $A^{n}$. 
 	\end{lemma}
 	\begin{proof}
 		The sequence of approximations of $A^n$ and $u^n_0$ satisfying the above properties can be constructed in the following ways.
 		\begin{enumerate}[Step-(i)]
 			\item \textit{Approximation of $u_0$ by piecewise constant functions:} Since initial data $u_0 \in BV(\R),$ we can approximate $u_0$ by a sequence of piecewise constant functions denoted by $u^{n}_0$ such that $TV(u_0^n)\leq TV(u_0)$. Therefore, by Lemma \ref{lemma:Psi} and \ref{lemma:app:pi} there exists a $C>0$ such that% satisfies 
 			$$TV(\Psi_A(\cdot,u_0^n))+TV(\pi_A(\cdot,u_0^n))\leq C$$
 			where $C$ depends on $TV(u_0)$, $TV(u_M^\pm)$ and the flux $A$. Hence \ref{a}, \ref{b} and \ref{c} hold. 
 			%we approximate $u_0$ by a sequence of piecewise constant functions denoted by $u^{n}_0$ such that \ref{a}, \ref{b} and \ref{c} hold. %having only finitely many discontinuities such that $u^{n}_0 \rightarrow u_0$  in $L^1_{loc}(\R)$ with $TV\left(u^n_0 \right) \leq TV\left(u_0\right)$ and if $m<u_0<M$ then $m<u^n_0<M$.
 			\item \textit{Approximate $A$ by a function $\tilde{A}^{n}$ having finitely many spatial discontinuities:} From assumption \ref{A4}, 
 			there exists a continuous function $\eta: \mathbb{R} \rightarrow \mathbb{R}$ and a BV function $a:\R \rightarrow \R$ such that $$|A(x,u)-A(y,u)| \leq \eta(u)|a(x)-a(y)| \quad \forall x,y,u \in \R.$$
 			Since $a\in BV(\R),$  corresponding to $a$ we define the following function $F(x):=TV_{(-\infty,x]}(a).$ Note that $F$ is an increasing function and left continuous function. We also observe that $F(x) \rightarrow TV_{\R}(a)$ as $x \rightarrow \infty$. Given $\epsilon >0,$ let $n$ be the largest integer such that $n\epsilon \leq TV_{\R}(a)$. For $j=1,2,...,n$ define the following
 			\begin{align*}
 			z_j&:=\min\{x: F(x)\geq j\epsilon\},\mbox{ for }1\leq j\leq n\mbox{ and }z_0=-\f,z_{n+1}=\f,\\
 			\tilde{A}^n(x,u)&:=A(z_j,u) \text{ for } x\in [z_{j},z_{j+1})\mbox{ for }1\leq j\leq l.%\mbox{ and }\tilde{A}^n(x,u):=A(z_n,u).
 			\end{align*}
 			For $x\in [z_i,z_{i+1})$, we get
 			\begin{eqnarray}
 			|\tilde{A}^n(x,u)-A(x,u)| &= & |A(z_i,u)-A(x,u)|\nonumber\\
 			&\leq & \eta(u) |a(z_i)-a(x)|\nonumber\\
 			&\leq & \eta(u)\epsilon.\label{cal1:A}
 			\end{eqnarray}
 			\item \textit{Approximation of $A$ by piecewise linear convex functions $A^n$ having finitely many spatial discontinuities whose break points form a complete set:}
 			Suppose $u^{n}_0$ takes values in the set $U^{n}$. Then define $P_n$ as the completion of the following set $$\left\{\frac{k}{2^n}: -2^n M\leq k \leq 2^n M, k\in \mathbb{Z}\right\}\bigcup U^{n}.$$ Due to Lemma \ref{lemma:completion} $P_n$ is a finite set.
 			For each $x \in \R$ approximate $\tilde{A^n}$ by piecewise linear convex function $A^n,$ such that $A^n$ has break up points only in $P_n.$ Clearly $A^n(x,u)=\tilde{A}^n(x,u)$ for $u\in P_n.$ Recall from previous step that for $\e>0$ we have $n\e\leq TV(a)\leq (n+1)\e$.
 			
 			\begin{eqnarray}
 			|A^n(x,u)-A(x,u)|%&\leq& |A^n(x,u)-\tilde{A}_n(x,u)-A(x,u)|\\
 			&\leq& |A^n(x,u)-\tilde{A}^n(x,u)|+|\tilde{A}^n(x,u)-A(x,u)|\nonumber\\
 			&\leq &C(M)\frac{1}{2^n}+\eta(u) \epsilon\nonumber\\
 			&\leq&C(M)\frac{1}{2^n}+\eta(u)\frac{TV_\R(a)}{n},
 			\end{eqnarray}
 			for $u\in[-M,M]$ with $M>0$.
 			\item For each $n$ define $k^{\pm}_{\alpha,n} $ by $A^{n}(x,k^{\pm}_{\alpha,n}(x))=\alpha$. Since $A^n \rightarrow A$ uniformly on compact sets and $u\mapsto A^{-1}(x,u),u\mapsto A^{-1}_n(x,u)$ are decreasing on $(-\f,u_M(x)]$ we get $k_{\alpha,n}^{\pm }(x) \rightarrow k_{\alpha}^{\pm}(x).$ 
 		\end{enumerate}
 	\end{proof}

 	The above construction assures the existence of approximations required in \eqref{FT2}. Now, we proceed to give the proof of the Theorem \ref{MT}
 	\begin{proof}[Proof of Theorem \ref{MT}:]
 		We split the proof into the following two steps.
 		\begin{description}
 		\descitem{Step-1}{existence:step-1} Consider a data ${u}_0\in BV(\R)$. Let $\{u^{\delta}\}$  be the sequence of Front Tracking approximations of \eqref{eqnap1}-\eqref{eqnap2}.
 		From Lemma \ref{lemma:tvd}, Lemma \ref{lemma:TVD:pi}, and the estimate \eqref{estimate:TVB-Psi-u}, there exists $C>0$ independent of $\delta$ such that,   
 		\[TV \left(\Psi^\de(u^{\delta} (\cdot,t)) \right)+TV \left(\pi^\de(u^{\delta} (\cdot,t))\right)\leq C, \quad \forall t\geq 0.\]
 		%For $m<u_0<M,$ we have  $$\sup\limits_{x,u \in \R}|\partial_u A(x,u)|\leq L$, speed of all the front are bounded by some $L>0.$ Thus we have 
 		Now applying Helly's theorem and Lemma \ref{timecontinuity} we get, up to a sub-sequence  $\Psi^\de(x,u^{\delta}(x,t))+\pi^\de(u^\de(x,t))$ converges pointwise a.e. (say to $\Phi$), i.e. $\Psi^\de(x,u^{\delta}(x,t))+\pi^\de(x,u^\de(x,t)) \rightarrow \Phi(x,t)$ a.e. in $\R \times [0, \infty).$  Hence, due to the invertibility of $\Psi(x,\cdot)+\pi(x,\cdot),$ there exists $u(x,t)$ such that $\Phi(x,t)=\Psi_{A}(x,u(x,t))+\pi_A(x,u(x,t))$ and subsequently, we have $u^{\delta} \rightarrow u$ a.e. in $\R \times \R^+.$ Since $A^{\delta} \rightarrow A$ pointwise,  $u$ is indeed a weak solution. From Lemma \ref{lemma:Riemann} the front tracking approximations $u^{\delta}_0$ satisfy adapted entropy corresponding to $A^{\delta}$ thus the limit satisfy the adapted entropy. Since every subsequential limit satisfies adapted entropy condition the whole sequence $u^{\delta}$ converges to the adapted entropy solution $u$ of the IVP \eqref{eqnap1}-\eqref{eqnap2}.
 		\descitem{Step-2}{existence:step-2} Now we prove the existence result for $L^\f$ data $u_0$. Let $\{u_0^n\}$ be an approximation such that $u_0^n\in BV(\R)$ and $u_0^n\rr u_0$ as $n\rr\f$. %By Lemma \ref{lemma:app:pi} we get a sequence $\{u_0^{n,k}\}\subset BV(\R)$ such that $TV(\pi(u_0^{n,k}))<\f$ and $u^{n,k}_0\rr u_0^n$ as $k\rr\f$. Hence, by Cantor's diagonalization argument we get a sequence of function $\{\tilde{u}_0^n\}\subset BV(\R)$ such that $\tilde{u}^n_0\rr u_0$ in $L^1_{loc}(\R)$ as $n\rr\f$ and $TV(\pi(\tilde{u}_0^n))<\f$ for each $n\geq1$. 
 		From \descref{existence:step-1}{Step-1} we have an adapted entropy solution ${u}^n$ corresponding to data ${u}_0^n$. Hence by estimate \eqref{estimate:stability} we have the existence of adapted entropy solution to the data $u_0$.
 		 		\end{description}
 	\end{proof}

 \section{Study of BV regularity}
 In this section we study the regularity of adapted entropy solutions in BV space. In the next subsection, we introduce a condition on initial data to get TV bound of entropy solution when the fluxes are uniformly convex. Later we give two counter-examples to show that neither the condition on initial data nor the uniform convexity assumption on flux can be relaxed.
  \subsection{Regularity in BV space}

  To obtain BV regularity of entropy solutions we need an assumption on initial data which is  stronger than BV. 
  Let $b\in L^\f(\R)$ and $-\f<x_1<\cdots<x_{n+1}<\f$. We define a sequence $\{\De w_j\}_{1\leq j\leq n}$ as follows,
  \begin{equation}\label{def:vj}
  \De w_j:=\left\{\begin{array}{ll}
  \abs{b(x_j)-b(x_{j+1})}&\mbox{ if }A(x_j,\cdot)= A(x_{j+1},\cdot),\\
  \abs{b(x_j)-u_M(x_j)}+\abs{b(x_{j+1})-u_M(x_{j+1})}&\mbox{ if }A(x_j,\cdot)\neq A(x_{j+1},\cdot),
  \end{array}\right.
  \end{equation}
  for $1\leq j\leq n$. We define $\Lambda(b)$ as follows,
  \begin{equation}\label{def:La}
  \Lambda(b):=\sup\left\{\sum\limits_{j=1}^{n}\De w_j;\,-\f<x_1<\cdots<x_{n+1}<\f\mbox{ and }\De w_j \mbox{ is defined as in \eqref{def:vj}}\right\}
  \end{equation}
  and the set $X:=\{b\in L^{\f}(\R);\,\Lambda(b)<\f\}$. Since $u_M\in BV(\R)$, we observe that $X\subset BV(\R)$. If $x\mapsto A(x,\cdot)$ is piece-wise constant then \eqref{def:La} can be simplified in the following way: let $-\f=z_0<z_1<\cdots<z_n<z_{n+1}=\f$ be discontinuity points of $x\mapsto A(x,u)$. We can write
  \begin{equation}
   \Lambda(b)=\sum\limits_{k=0}^{n}TV_{(z_k,z_{k+1})}(b)+\sum\limits_{k=1}^{n}\left[\abs{b(z_k-)-u_M(z_k-)}+\abs{b(z_k+)-u_M(z_k+)}\right].
  \end{equation} 
 % -----------------------------------------------------------------
  \begin{theorem}\label{theorem:bv}
 	Let $A(\cdot,\cdot)$ satisfy \eqref{A1}--\eqref{A3} and \eqref{A4}--\eqref{A5} with $u_M^+(x)=u_M^-(x)$ for $x\in\R$. Additionally, we assume that $A(\cdot,x)\in C^2(\R)$ and $\pa_{uu}A(\cdot,x)\geq C_1$ for all $x\in\R$ where $C_1>0$ does not depend on $x$. Let $u$ be an adapted entropy solution to \eqref{eq:discont}--\eqref{eq:data} with initial data $u_0$ such that $\Lambda(u_0)<\f$ where $\Lambda $ is defined as in \eqref{def:La}. Then we have	
 	\begin{equation}
 	TV(u(\cdot,t))\leq TV(u_M)+C_3\Lambda(u_0)\mbox{ for }t>0,
 	\end{equation}
 	$\mbox{ where }C_2=\sup\{\|A(\cdot,x)\|_{C^2[-M,M]},\,x\in\R\}$ and $C_3:=\left(\frac{C_2}{C_1}\right)^{\frac{3}{2}}$.
 \end{theorem}
 Before the proof of Theorem \ref{theorem:bv}, here we prove a property of uniform convex flux in the following lemma. 
 \begin{lemma}\label{lemma:uniformly-convex}
 	Let $f,g\in C^2(I)$ such that $f^{\p\p},g^{\p\p}\in[C_1,C_2]$ for $C_2>C_1>0$.  Additionally, assume that $f(\theta_f)=g(\theta_g)=0$. Then we have
 	\begin{equation}
 	\abs{g_+^{-1}(a)-g_{+}^{-1}(b)}\leq C_3\abs{f_+^{-1}(a)-f_+^{-1}(b)}\mbox{ for }a,b\in I\mbox{ where }C_3=\left(\frac{C_2}{C_1}\right)^{\frac{3}{2}}.
 	\end{equation}
 \end{lemma}
 \begin{proof}
 	Suppose $f(u_1)=g(u_2)$ with $u_1>\theta_f,u_2>\theta_g$. Now, we observe that
 	\begin{align}
 	f(u_1)&=f(\theta_f)+f^{\p}(\theta_f)(u_1-\theta_f)+\frac{f^{\p\p}(u_*)}{2}\abs{u_1-\theta_f}^2\geq \frac{C_1}{2}\abs{u_1-\theta_f}^2,\\
 	g(u_2)&=g(\theta_g)+g^{\p}(\theta_g)(u_2-\theta_g)+\frac{g^{\p\p}(u^*)}{2}\abs{u_2-\theta_g}^2\leq \frac{C_2}{2}\abs{u_2-\theta_g}^2.
 	\end{align}
 	Subsequently, we have
 	\begin{equation}
 	\frac{C_1}{2}\abs{u_1-\theta_f}^2\leq \frac{C_2}{2}\abs{u_2-\theta_g}^2.
 	\end{equation}
 	Therefore, we have
 	\begin{equation}
 	(u_1-\theta_f)\leq \left(\frac{C_2}{C_1}\right)^{\frac{1}{2}}(u_2-\theta_g).
 	\end{equation}
 	Notice that
 	\begin{align*}
 	f^{\p}(f^{-1}_{+}(p))=f^{\p}(f^{-1}_{+}(p))-f^{\p}(\theta_f)&=(f^{-1}_{+}(p)-\theta_f)\int\limits_{0}^{1}f^{\p\p}(\la f^{-1}_{+}(p)+(1-\la)\theta_f)\,d\la\\
 	&\leq \left(\frac{C_2}{C_1}\right)^{\frac{3}{2}}(g^{-1}_{+}(p)-\theta_g)\int\limits_{0}^{1}g^{\p\p}(\la g^{-1}_{+}(p)+(1-\la)\theta_g)\,d\la\\
 	&=\left(\frac{C_2}{C_1}\right)^{\frac{3}{2}}g^{\p}(g^{-1}_{+}(p)).
 	\end{align*}
 	Hence, we have
 	\begin{align*}
 	g^{-1}_+(a)-g^{-1}_+(b)&=(a-b)\int\limits_{0}^{1}\frac{1}{g^{\p}(g^{-1}_+(\la a+(1-\la)b))}\,d\la\\
 	&\leq \left(\frac{C_2}{C_1}\right)^{\frac{3}{2}}(a-b)\int\limits_{0}^{1}\frac{1}{f^{\p}(f^{-1}_+(\la a+(1-\la)b))}\,d\la\\
 	&=\left(\frac{C_2}{C_1}\right)^{\frac{3}{2}}(f^{-1}_+(a)-f^{-1}_+(b)).
 	\end{align*}
 \end{proof}

 \begin{proof}[Proof of Theorem \ref{theorem:bv}]
 	We prove the estimate for approximated solution $u^\de$ and then, we show that estimates are uniform over $\de$. Let $u^\de$ be approximated solution to \eqref{eq:discont} obtained via wave front tracking. Note that there is no affine part on $A(\cdot,x)$ for all $x\in\R$ due to $\pa_{uu}A(u,x)\geq C_1$. By using Lemma \ref{lemma:contact-dis}, no new state arises away from interface for time $t>0$. Since there is no linear part in the flux, no rarefaction wave occurs at the interface for $t>0$. Therefore, for a time $t_0>0$ and $x\in\R$, there exists $y(x)\in\R$ such that $A^\de(u^\de(x,t_0),x)=A^\de(u_0^\de(y(x)),y(x))$. Suppose $-\f=x_{0}<x_{1}<\cdots<x_{m}<x_{m+1}=\f$ are discontinuity points of $x\mapsto A(\cdot,x)$. Let $z_1,z_2\in(x_i,x_{i+1})$ for some $0\leq i\leq m$. Now, there are two cases
 	\begin{enumerate}
 		\item $y(z_1),y(z_2)\in (x_j,x_{j+1})$ for some $0\leq j\leq m$. Then by Lemma \ref{lemma:uniformly-convex} we have 
 		\begin{equation}\label{estimate1}
 		\abs{u^\de(z_1,t_0)-u^\de(z_2,t_0)}\leq C_3 \abs{u^\de_0(y(z_1))-u^\de_0(y(z_2))}.
 		\end{equation}
 		\item $y(z_1)\in (x_j,x_{j+1}),y(z_2)\in(x_{k},x_{k+1})$ for some $0\leq j\neq k\leq m$. Then by Lemma \ref{lemma:uniformly-convex} and triangle inequality, we have 
 		\begin{align}
 		\abs{u^\de(z_1,t_0)-u^\de(z_2,t_0)}&\leq 	\abs{u^\de(z_1,t_0)-u_M^\de(z_1)}+\abs{u^\de(z_2,t_0)-u_M^\de(z_2)}+\abs{u_M^\de(z_1)-u_M^\de(z_2)}\nonumber\\
 		&\leq C_3\left[\abs{u^\de_0(y(z_1))-u_M^\de(y(z_1))}+\abs{u^\de_0(y(z_2))-u^\de_M(y(z_2))}\right]\nonumber\\
 		&+\abs{u_M^\de(z_1)-u_M^\de(z_2)}.\label{estimate2}
 		\end{align}
 	\end{enumerate}
 	Suppose $z\in(x_{i},x_{i+1})$ for $0\leq i\leq m-1$, then by Lemma \ref{lemma:uniformly-convex}, we have
 	\begin{equation}\label{estimate3}
 	\abs{u^\de(z,t_0)-u^\de_M(z)}\leq C_3 \abs{u^\de_0(y(z))-u^\de_M(y(z))}.
 	\end{equation}
 	Combining \eqref{estimate1}, \eqref{estimate2} and \eqref{estimate3} with definition \eqref{def:La} of $\Lambda$ we conclude Theorem \ref{theorem:bv}.
 \end{proof}
 
 \subsection{Counterexample-I}\label{sec:counter-ex-1}
 In this section, we wish to construct an example of adapted entropy solution to a flux satisfying \ref{A1}--\ref{A3} and \ref{A4}--\ref{A5} such that $u(\cdot,1)\notin BV_{loc}(\R)$ for a data $u_0\in BV(\R)$. Here we consider a flux which is additionally satisfying $C^2$ regularity, that is, for a.e. $x\in\R$, $u\mapsto A(x,u)$ is a $C^2$-function. Note that this result is special for fluxes having infinitely many spatial discontinuities since for a flux $A$ with finitely many discontinuity and satisfying \ref{A1}--\ref{A3} and \ref{A4}--\ref{A5}, $TV(u(\cdot,t))<\f$ for BV initial data (see \cite{Ghoshal-JDE}). 
 \begin{proposition}
 	There exists a flux $A$ satisfying \eqref{A1}--\eqref{A3}, \eqref{A4}--\eqref{A5} and an initial data $u_0$ such that $\La(u_0)<\f$ such that the corresponding entropy solution to \eqref{eq:discont} has TV blow up at some finite time $T_0>0$.
 \end{proposition}
 
 \noindent\textbf{Construction:} Let us consider two sequences $\{y_n\}_{n\geq0}$ and $\{z_n\}_{n\geq 1}$ such that they satisfy
 \begin{equation}\label{def:seq}
 y_{n-1}<z_n<y_n\mbox{ for }n\geq1\mbox{ and } \lim\limits_{n\rr\f}y_n=\lim\limits_{n\rr\f}z_n=z^*\mbox{ for some }z^*\in\R.
 \end{equation}
 Now we consider the following flux,
 \begin{equation}\label{def:F}
 A(x,u):=\sum\limits_{n=1}^{\f}f_n(u)\chi_{(y_{n-1},z_n]}(x)+\sum\limits_{n=1}^{\f}g_n(u)\chi_{(z_n,y_n]}(x)+f_1\chi_{(-\f,y_0]}(x)+u^2\chi_{(z^*,\f]}(x)
 \end{equation}
 where $f_n,g_n$ are defined as follows
 \begin{equation}
 f_n(u)=\frac{2}{n^{3/2}}u^2+u^4\mbox{ and }g_n(u)=u^4.
 \end{equation}
 Note that $u\mapsto A(x,u)$ is a $C^2$-function for a.e. $x\in\R$. Hence \ref{A1} is satisfied. Observe that $A(x,u)$ satisfies \ref{A2}, \ref{A3} and \ref{A5} with $B_1(u)=u^4,B_2(u)=2u^4$ and $u_M\equiv0$. Note that \ref{A4} is verified for $A(x,u)$ with the following choice of $\eta,a$
 \begin{equation}
 \eta(u):=2u^2\mbox{ and }a(x):=\sum\limits_{n=1}^{\f}\frac{1}{n^{\frac{3}{2}}}\chi_{(y_{n-1},z_n]}(x)+1\chi_{(-\f,y_0]}(x).
 \end{equation}
 Loosely speaking, $\abs{(g_n)_+^{-1}(f_n(u))}\geq u^{\frac{1}{2}}n^{-\frac{3}{8}}$ gives a possible blow up in the solution.
 
 We consider an initial data defined as follows,
 \begin{equation}
 u_0^n(x):=\left\{\begin{array}{rl}
 0&\mbox{ if }y_{n-1}<x<x_n^1,\\
 a_n&\mbox{ if }x_n^1<x<x_n^2,\\
 0&\mbox{ if }x_n^2<x<z_n,\\
 0&\mbox{ if }z_n<x<y_n.
 \end{array}\right.
 \end{equation}
 We observe that rarefaction wave arises at $x_n^1$ and shock wave is generated at $x_n^2$ with shock speed $\la_n$ where $\la_n$ is calculated as follows
 \begin{equation}\label{def:la}
 \la_n=\frac{f_n(a_n)-f_n(0)}{a_n-0}=\frac{2a_n}{n^{\frac{3}{2}}}+a_n^3.
 \end{equation}
 Let at time $t=t_n$ the shock curve hits the interface, therefore, we have $x_n^2+t_n\la_n=z_n$. Since the speed of extreme right characteristic of rarefaction wave $f^{\p}_n(a_n)$ is bigger than $a_n$, we need to give sufficient gap between $x_n^1$ and $x_n^2$ such that before time $t=1$ it does not hit the shock wave and the interface at $z_n$. Hence, it must hold 
 \begin{equation}\label{example:condition1}
 z_n-x_n^1\geq f^{\p}_n(a_n).
 \end{equation}
 Now after the shock wave from $x_n^2$ hits the interface at $z_n$ it generates another shock wave from $(z_n,t_n)\in\R\times\R_+$. Note that this new shock wave has states $b_n,0$ on its left and right respectively where $b_n=(g_n)_+^{-1}(f_n(a_n))$. We want this shock wave not to meet interface at $y_n$ before time $t=1$, hence it must hold, 
 \begin{equation}\label{example:condition2}
 \xi_n(1-t_n)+z_n<y_n,
 \end{equation} where $\xi_n$ is the shock speed between $b_n,0$, that is,
 \begin{equation}\label{def:xi}
 \xi_n=\frac{g_n(b_n)-g_n(0)}{b_n-0}=b_n^3.
 \end{equation}
 If we are able to choose such $x_n^1,x_n^2,z_n,y_n$ then the solution has the following structure up to time $t=1$, (see figure \ref{figure:example})

 %\noi{\color{blue}May be blow up can come}
 \begin{enumerate}
 	
 	\item For $t\in[0,t_n]$ we have
 	\begin{equation}
 	u^n(x,t):=\left\{\begin{array}{rl}
 	0&\mbox{ if }x<x_n^1,\\
 	(f_n^{\p})^{-1}(x-x_{n}^1/t)&\mbox{ if }x_n^1<x<x_n^1+f_n^{\p}(a_n)t,\\
 	a_n&\mbox{ if }x_n^1+f_n^{\p}(a_n)t<x<x_n^2+\la_nt,\\
 	0&\mbox{ if }x_n^2+\la_nt<x<z_n,\\
 	0&\mbox{ if }z_n<x<y_n,
 	\end{array}\right.
 	\end{equation}
 	where $\la_n$ is defined as in \eqref{def:la}.
 	
 	\item For $t\in[t_n,1]$ we have
 	
 	\begin{equation}
 	u^n(x,t):=\left\{\begin{array}{rl}
 	0&\mbox{ if }x<x_n^1,\\
 	(f_n^{\p})^{-1}(x-x_{n}^1/t)&\mbox{ if }x_n^1<x<x_n^2+f_n^{\p}(a_n)t,\\
 	a_n&\mbox{ if }x_n^2+f_n^{\p}(a_n)t<x<z_n,\\
 	b_n&\mbox{ if }z_n<x<z_n+\xi_n(t-t_n),\\
 	0&\mbox{ if }z_n+\xi_n(t-t_n)<x<y_n
 	\end{array}\right.
 	\end{equation}
 	where $\xi_n $ is defined as in \eqref{def:xi} and $b_n$ is defined as follows,
 	\begin{equation}
 	b_n=(g_n)_+^{-1}(f_n(a_n)).
 	\end{equation}	
 \end{enumerate}
 Finally, we choose 
 \begin{equation}
 a_n=\frac{1}{n^{1+\de}},\, x_n^2-x_n^1=\frac{8}{n^{\de+\frac{5}{2}}}+\frac{8}{n^{3+3\de}}\mbox{ and }z_n-x_n^2=\left(\frac{4}{n^{\de+\frac{5}{2}}}+\frac{2}{n^{3+3\de}}\right)\left(1-\frac{1}{n^2}\right),
 \end{equation}
 where $\de>0$.
 Next we choose
 \begin{equation}
 y_n-z_n=2\left(\frac{2}{n^{2\de+\frac{7}{2}}}+\frac{1}{n^{4+4\de}}\right)^{\frac{3}{4}}\leq 6\left(\frac{1}{n^{2\de+\frac{7}{2}}}\right)^{\frac{3}{4}}\mbox{ and }x_n^1=y_{n-1}+\frac{1}{n^2}.
 \end{equation}
 Note that
 \begin{equation}
 y_{n-1}<x_n^1<x_n^2<z_n<y_n.
 \end{equation}
 Note that $x_n^2-x_n^1=f^{\p}_n(a_n)$. Hence, we have $z_n-x_n^1=z_n-x_n^2+x_n^2-x_n^1>f^{\p}(a_n)$ since $z_n-x_n^2>0$. Therefore, \eqref{example:condition1} is satisfied. Observe that
 $y_n-z_n=2((g_n)_+^{-1}(f_n(a_n)))^{3}=2\xi_n>\xi_n(1-t_n)$ since $\xi_n>0$ and $0<1-t_n<1$. Hence, the condition \eqref{example:condition2} is satisfied. Now we note that
 \begin{align}
 y_n-y_{n-1}&=y_n-z_n+z_n-x_n^2+x_n^2-x_n^1+x_n^1-y_{n-1}\\
 &\leq 6\left(\frac{1}{n^{2\de+\frac{7}{2}}}\right)^{\frac{3}{4}}+\left(\frac{4}{n^{\de+\frac{5}{2}}}+\frac{2}{n^{3+3\de}}\right)\left(1-\frac{1}{n^2}\right)+\frac{8}{n^{\de+\frac{5}{2}}}+\frac{8}{n^{3+3\de}}+\frac{1}{n^2}\\
 &\leq \frac{1}{n^2}+\frac{22}{n^{\de+\frac{5}{2}}}+\frac{8}{n^{\frac{21}{8}}}.
 \end{align}
 Thus, there exists an $y_{\f}$ such that $y_n\rr y_{\f}$. Set $z^*=y_{\f}$ and $y_0=0$ in \eqref{def:F}.
 
 Note that $TV(u_0^n)\leq n^{-1-\de}$ and $TV(u^n(\cdot,1))\geq \abs{b_n}$. 
 Observe that
 \begin{align}
 \abs{b_n}=\left[\frac{2}{n^{2\de+\frac{7}{2}}}+\frac{1}{n^{4+4\de}}\right]^{\frac{1}{4}}\geq \frac{1}{n^{\frac{7}{8}+\frac{\de}{2}}}.
 \end{align}

 Note if we take an initial data $u_0=\sum\limits_{k=1}^{\f}u_0^k$ then, $u=\sum\limits_{k=1}^{\f}u_k$ will be adapted entropy solution to \eqref{eq:discont}. Since $TV(u_0^k)\leq k^{-1-\de}$ and $supp(u_0^k),k\geq1$ are mutually disjoint we have $TV(u_0)<\f$. We observe that $supp(u^k(\cdot,1)),k\geq1$ are mutually disjoint. Therefore, we have,
 \begin{equation}
 TV(u(\cdot,1))\geq \sum\limits_{k\geq1}\frac{1}{k^{\frac{7}{8}+\frac{\de}{2}}}.
 \end{equation}
 Now we choose, $\de\in(0,4^{-1})$. Then, we have $\frac{7}{8}+\frac{\de}{2}<1$. Hence, $TV(u(\cdot,1))=\f$.
 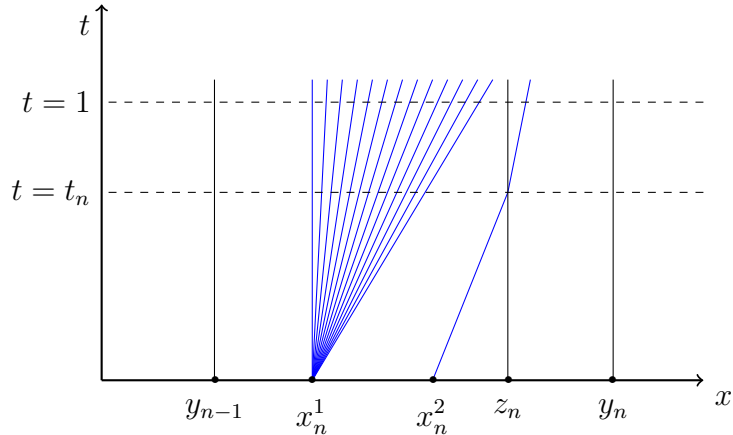
\begin{figure}[ht]
 	\centering
 	\begin{tikzpicture}

 	\draw[thick,->] (-9 ,0) -- (-1,0) node[anchor=north west] {$x$};
 	\draw[thick,->] (-9 ,0) -- (-9,5) node[anchor=north east] {$t$};
 	\draw[color=black] (-7.5 ,0) -- (-7.5,4) ;
 	\draw[color=blue] (-6.2 ,0) -- (-6.2,4) ;
 	\draw[color=blue] (-6.2 ,0) -- (-6,4) ;
 	\draw[color=blue] (-6.2 ,0) -- (-5.8,4) ;
 	\draw[color=blue] (-6.2 ,0) -- (-5.6,4) ;
 	\draw[color=blue] (-6.2 ,0) -- (-5.4,4) ;
 	\draw[color=blue] (-6.2 ,0) -- (-5.2,4) ;
 	\draw[color=blue] (-6.2 ,0) -- (-5,4) ;
 	\draw[color=blue] (-6.2 ,0) -- (-4.8,4) ;
 	\draw[color=blue] (-6.2 ,0) -- (-4.6,4) ;
 	\draw[color=blue] (-6.2 ,0) -- (-4.4,4) ;
 	\draw[color=blue] (-6.2 ,0) -- (-4.2,4) ;
 	\draw[color=blue] (-6.2 ,0) -- (-4,4) ;
 	\draw[color=blue] (-6.2 ,0) -- (-3.8,4) ;
 	\draw[color=blue] (-4.6 ,0) -- (-3.6,2.5) ;
 	\draw[color=blue] (-3.6 ,2.5) -- (-3.3,4) ;
 	\draw[color=black] (-3.6 ,0) -- (-3.6,4) ;
 	\draw[color=black] (-2.2 ,0) -- (-2.2,4) ;
 	
 	\draw[thick][] (-5.21,-0.18) node[anchor=south west, transform canvas={scale=1.5}]{$\textbf{.} $};			
 	
 	\draw[thick][] (-4.35,-0.18) node[anchor=south west, transform canvas={scale=1.5}]{$\textbf{.} $};
 	\draw[thick][] (-3.28,-0.18) node[anchor=south west, transform canvas={scale=1.5}]{$\textbf{.} $};
 	\draw[thick][] (-2.61,-0.18) node[anchor=south west, transform canvas={scale=1.5}]{$\textbf{.} $};
 	\draw[thick][] (-1.69,-0.18) node[anchor=south west, transform canvas={scale=1.5}]{$\textbf{.} $};
 	\draw[thick][] (-6.2,-.1) node[anchor=north] {$x^1_n$};
 	\draw[thick][] (-4.6,-.1) node[anchor=north] {$x^2_{n}$};
 	\draw[thick][] (-3.6,-.1) node[anchor=north] {$z_n$};
 	\draw[thick][] (-7.5,-.1) node[anchor=north] {$y_{n-1}$};
 	\draw[thick][] (-2.2,-.1) node[anchor=north] {$y_n$};
 	\draw[dashed] (-1,2.5) -- (-9,2.5) node[anchor= east] {$t=t_n$};
 	\draw[dashed] (-1,3.7) -- (-9,3.7) node[anchor= east] {$t=1$};
 	\end{tikzpicture}
 	\caption{This illustrates the structure of the solution in $[y_{n-1},y_n]\times[0,1]$. This is the building block for the solution giving TV blow up at time $t=1$ for strictly convex flux. Here $x=y_{n-1}$, $x=z_n$, $x=y_{n}$  represent interfaces. Note that rarefaction and shock are generated at $(x^1_n,0)$ and $(x^2_n,0)$ respectively. The shock line hits the interface $x=z_n$ at time $t=t_n<1$.}\label{figure:example}
 \end{figure}

\subsection{Counterexample-II}\label{sec:counter-ex-2}
In section \ref{sec:counter-ex-1} we show a BV blow-up of entropy solutions with the quantity $\La(u_0)$ is finite for initial data $u_0$ where the fluxes $u\mapsto A(x,u)$ is not necessarily uniformly convex. In this section we construct an example coming from BV data and uniformly convex flux such that it has a TV blow-up at some finite time $T$. In this situation we construct a flux and data for which the quantity $\La(u_0)$ is not finite. 

 \begin{proposition}
	There exists a flux $A$ satisfying \eqref{A1}--\eqref{A3}, \eqref{A4}--\eqref{A5} and an initial data $u_0$ such that $\pa_{uu}A(x,u)\in[C_1,C_2]$ for some $C_2>C_1>0$ and the corresponding entropy solution to \eqref{eq:discont} has TV blow up at some finite time $T_0>0$.
\end{proposition}

\noindent\textbf{Construction:} We split the construction into several steps.

\begin{description}
\descitem{Step-1}{step-A}\textit{Construction of flux and initial data:} In this very first step we construct a flux function $A$ satisfying conditions \eqref{A1}--\eqref{A3} and \eqref{A4}--\eqref{A5}. Note that due to the condition \eqref{A5}, we need to construct a flux function which is BV in spatial variable that is why the choice of flux becomes tricky. Then we choose a suitable initial data.

To construct the flux $A(x,u)$, we define functions $g,f_n$ as follows,
\begin{equation*}
f_n(u):=\left\{\begin{array}{ll}
u^2&\mbox{ if }u\leq -\frac{1}{n^{2/3}},\\
u^2+\frac{u^2}{n^{1/4}}-3n^{13/12}u^4+3n^{29/12}u^6-n^{15/4}u^8&\mbox{ if }-\frac{1}{n^{2/3}}\leq u\leq \frac{1}{n^{2/3}},\\
%n^{4/3}u^4-\frac{1}{3}n^{8/3}u^6+u^2&\mbox{ if }0\leq u\leq \frac{1}{n^{2/3}},\\
u^2&\mbox{ if }u\geq \frac{1}{n^{2/3}},
\end{array}\right.
\end{equation*}
and $g(u)=u^2$. After a calculation we get
\begin{equation*}
f_n^{\p}(u)=\left\{\begin{array}{ll}
2u&\mbox{ if }u\leq -\frac{1}{n^{2/3}},\\
2u+\frac{2u}{n^{1/4}}-12n^{13/12}u^3+18n^{29/12}u^5-8n^{15/4}u^7&\mbox{ if }-\frac{1}{n^{2/3}}\leq u\leq \frac{1}{n^{2/3}},\\
2u&\mbox{ if }u\geq \frac{1}{n^{2/3}}.
\end{array}\right.
\end{equation*} and
\begin{equation*}
f_n^{\p\p}(u)=\left\{\begin{array}{ll}
2&\mbox{ if }u\leq -\frac{1}{n^{2/3}},\\
2+\frac{2}{n^{1/4}}-36n^{13/12}u^2+90n^{29/12}u^4-56n^{15/4}u^6&\mbox{ if }-\frac{1}{n^{2/3}}\leq u\leq \frac{1}{n^{2/3}},\\
2&\mbox{ if }u\geq \frac{1}{n^{2/3}}.
\end{array}\right.
\end{equation*}
Hence $f_n\in C^2(\R)$ for $n\geq 1$. Note that there exists $n_0\in\mathbb{N}$ such that $f^{\p\p}_n(u)\geq 1$ for $n\geq n_0$. Suppose $\{a_n\},\{b_n\}$ are two sequences such that $0<a_{n+1}< b_{n+1}<a_n$. We consider a flux $A$ as follows
\begin{equation*}
A(x,u):=\left\{\begin{array}{rl}
f_n(u)&\mbox{ for }a_n\leq x\leq b_n\mbox{ for }n\geq n_0,\\
g(u)&\mbox{ otherwise.}
\end{array}\right.
\end{equation*}
We choose $a_n=\frac{1}{n^{2/3}}$ and $b_n=a_n+\frac{1}{n^{2n}}$. We approximate $A$ by $A^N$ as follows
\begin{equation*}
A^N(x,u):=\left\{\begin{array}{rl}
f_n(u)&\mbox{ for }a_n\leq x\leq b_n\mbox{ with }n_0\leq n\leq N,\\
g(u)&\mbox{ otherwise.}
\end{array}\right.
\end{equation*}
By the choice of $f_n(u)=g(u)$ for $\abs{u}\geq n^{-2/3}$. Consider the following data
\begin{equation}\label{def:counter-data}
u_0(x):=\left\{\begin{array}{rl}
0&\mbox{ for }x<0,\\
1&\mbox{ for }x>0.
\end{array}\right.
\end{equation}

\descitem{Step-2}{step-B}\textit{Analysis of characteristics:} In this step we track the paths travelled by the characteristics emanates from $x=0,t=0$ to get structural information about adapted entropy solution to \eqref{eq:discont} corresponding to flux $A^N$ and data $u_0$ as in \descref{step-A}{Step-1}.

Let $u^N$ be the adapted entropy solution to \eqref{eq:discont} with flux $A^N$ and data $u_0$. Note that initially rarefaction arises at $x=0$ and then it travels through interfaces and there is no interaction between waves. We observe that since $n_0\geq 4$, $f_n=g=u^2$ for $u=1$, hence no wave arises from interfaces at $t=0+$ for $x>0$. Let $L_N:[0,2]\rr\R$ be the path travelled by the characteristic arising from $x=0$ with speed $2w_N$ with $w_N:=(N+1)^{-2/3}$. Rest of this step is divided into two parts.

\begin{description}
	\descitem{Step-2a}{step-2a}\textit{Analyzing the curve $L_N$:} Note that $L_N$ is a piecewise linear curve originated from $x=0$ (see figure \ref{figure:example-2}). We get $L_N(0)=0$ and wish to estimate $L_N(1)$.
	 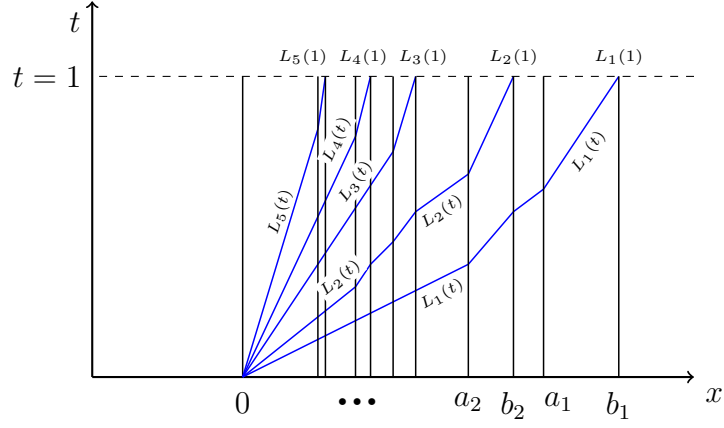
\begin{figure}
		\centering
		\begin{tikzpicture}

		\draw[thick,->] (-9 ,0) -- (-1,0) node[anchor=north west] {$x$};
		\draw[thick,->] (-9 ,0) -- (-9,5) node[anchor=north east] {$t$};
		
		\draw[line width=0.2mm, black] (-7 ,0) -- (-7,4) ;
		\draw[line width=0.2mm, blue] (-7 ,0) -- (-4,1.5)-- (-3.4,2.2)-- (-3,2.5) -- (-2,4) ;
		\draw[line width=0.2mm, blue] (-7 ,0) -- (-5.5,1.2) -- (-5.3,1.5) -- (-5,1.8) -- (-4.7,2.2) -- (-4,2.7) -- (-3.4,4) ;
		\draw[line width=0.2mm, blue] (-7 ,0) -- (-5,3) -- (-4.7,4) ;
		\draw[line width=0.2mm, blue] (-7 ,0) -- (-5.5,3.2) -- (-5.3,4) ;
		\draw[line width=0.2mm, blue] (-7 ,0) -- (-6,3.3) -- (-5.9,4) ;
		%			\draw[color=blue] (-7 ,0) -- (-6.5,3.4) -- (-6.3,4) ;
		\draw[line width=0.2mm, black] (-3,0) -- (-3,4);
		\draw[line width=0.2mm, black] (-2,0) --(-2,4) ;
		\draw[line width=0.2mm, black] (-4 ,0) -- (-4,4);
		\draw[line width=0.2mm, black] (-3.4,0) -- (-3.4,4) ;
		\draw[line width=0.2mm, black] (-5 ,0) -- (-5,4) ;
		\draw[line width=0.2mm, black] (-4.7,0)--(-4.7,4);
		\draw[line width=0.2mm, black] (-5.5 ,0) -- (-5.5,4) ;
		\draw[line width=0.2mm, black] (-5.3,0)--(-5.3,4);
		\draw[line width=0.2mm, black] (-6 ,0) -- (-6,4) ;
		\draw[line width=0.2mm, black] (-5.9,0)--(-5.9,4);

		\draw[thick][] (-2,-.05) node[anchor=north] {$b_1$};
		\draw[thick][] (-2.8,-.05) node[anchor=north] {$a_{1}$};
		\draw[thick][] (-3.4,-.05) node[anchor=north] {$b_{2}$};
		\draw[thick][] (-4,-.05) node[anchor=north] {$a_{2}$};
		\draw[thick][] (-4,-.4) node[anchor=south west, transform canvas={scale=1.5}]{$\textbf{...} $};
		\draw[thick][] (-7,-.05) node[anchor=north] {$0$};
		
		\draw[dashed] (-1,4) -- (-9,4) node[anchor= east] {$t=1$};
	
	%	\draw (1.5,2.5)node{\tiny $(f')^{-1}\left(\frac{x-\xi_{0}^{+}}{t}\right)$}
		
%			\draw[thick][] (-2.4,3) node[anchor=north] {\tiny $L_1(t)$};
%			\draw[thick][] (-4.35,1.2) node[anchor=north] {\tiny $L_1(t)$};
%		\draw[thick][] (-4.35,2.4) node[anchor=north] {\tiny $L_2(t)$};
%		
		\draw[thick][] (-2,4) node[anchor=south] {\tiny $L_1(1)$};
		\draw[thick][] (-3.4,4) node[anchor=south] {\tiny $L_2(1)$};
		\draw[thick][] (-4.6,4) node[anchor=south] {\tiny $L_3(1)$};
		\draw[thick][] (-5.4,4) node[anchor=south] {\tiny $L_4(1)$};
		\draw[thick][] (-6.2,4) node[anchor=south] {\tiny $L_5(1)$};

		\draw (-2.4,3) node[rotate=55](N){\tiny$L_{1}(t)$};
		\draw (-4.35,1.1) node[rotate=30](N){\tiny$L_{1}(t)$};
		\draw (-4.35,2.2) node[rotate=35](N){\tiny$L_{2}(t)$};

	%		\filldraw[white!40!white] (-5.8,1.5) rectangle (-5.5,1.8) ;	
					\draw (-5.7,1.32) node[fill=white, inner sep=1pt, rotate=40](N){\tiny$L_{2}(t)$};
			\draw (-5.75,3.16) node[fill=white, inner sep=1pt,rotate=67](N){\tiny$L_{4}(t)$};	
		\draw (-5.5,2.6) node[fill=white, inner sep=1pt,rotate=55](N){\tiny$L_{3}(t)$};	
			\draw (-6.5,2.2) node[fill=white, inner sep=1pt,rotate=73](N){\tiny$L_{5}(t)$};
		\end{tikzpicture}
		\caption{This figure shows the structure of the entropy solution for approximated flux $A^{N+5}$. Here the curves $L_k(t)$ represent the path traveled by characteristic arising from origin with speed $2(N+k)^{-2/3}$ for $k=1,2,3,4,5$. }\label{figure:example-2}
	\end{figure}

	 Let $v_{N,k}\in[0,\f)$ such that $g(w_N)=f_m(v_{N,k})$. Suppose $L_N(t)\in [a_k,b_k]$ for $t\in[t_k,t_k']$. Therefore, we have $f^{\p}_k(v_{N,k})(t_k'-t_k)=b_k-a_k$. Suppose $t^{\p\p}_k>t_k$ such that $b_k-a_k=(t^{\p\p}_k-t_k)g^{\p}(w_N)$. We obtain
\begin{equation*}
\abs{t^{\p\p}_k-t^{\p}_k}=(b_k-a_k)\frac{\abs{f^{\p}_k(v_{N,k})-g^{\p}(w_N)}}{f^{\p}_k(v_{N,k})g^{\p}(w_N)}.
\end{equation*}
Subsequently, we have
\begin{equation*}
g^{\p}(w_N)\abs{t^{\p\p}_k-t^{\p}_k}=(b_k-a_k)\frac{\abs{f^{\p}_k(v_{N,k})-g^{\p}(w_N)}}{f^{\p}_k(v_{N,k})}.
\end{equation*}
Let $L(\tilde{t}_N)=a_{N/2}$, then we have
\begin{equation*}
\abs{1-\tilde{t}_{N}}\leq \sum\limits_{k=\left[\frac{N}{2}\right]+1}^N(b_k-a_k)\frac{\abs{f^{\p}_k(v_{N,k})-g^{\p}(w_N)}}{f^{\p}_k(v_{N,k})g^{\p}(w_N)}.
\end{equation*}
Note that $\frac{1}{(N+1)^{2/3}}\geq \frac{1}{2(N/2)^{2/3}}$ for $N\geq n_1$ for some $n_1>1$. Subsequently, we have $\frac{k^{2/3}}{(N+1)^{2/3}}\geq \frac{1}{2}$. Now we observe that 
\begin{align*}
f_k\left(\frac{1}{2(N+1)^{2/3}}\right)&=\frac{1}{2(N+1)^{4/3}}\left(1+\frac{1}{k^{1/4}}\left(1-\frac{k^{4/3}}{4(N+1)^{4/3}}\right)^3\right)\\
&\leq g(w_N)\frac{1}{2^{4/3}}(1+\frac{1}{16k^{1/4}})\\
&\leq g(w_N)=f_k(v_{N,k}).
\end{align*}
Hence, we get $\frac{w_N}{2}\leq v_{N,k}$ for $N/2+1\leq k\leq N$. Since $(1-k^{4/3}u^2)^3\geq 0$ for $\abs{u}\leq k^{-2/3}$ we have $f_k(w_N)\geq f_k(v_{N,k})$, therefore $w_N\geq v_{N,k}$. 
%From Lemma \ref{lemma:alg} we have $w_N\left(1+\frac{c}{k^{1/4}}\right)\leq f^{\p}_k(w_N)\leq w_N\left(1+\frac{C}{k^{1/4}}\right)$. We also note that $0<v_{N,k}\leq w_N\leq v_{N,k}\left(1+\frac{B}{k^{1/4}}\right)$. 
This shows that 
\begin{equation*}
\abs{f^{\p}(v_{N,k})-g^{\p}(w_N)}\leq \abs{w_N-v_{N,k}}+\frac{\bar{C}}{k^{1/4}}v_{N,k}\leq {\tilde{C}}w_N.
\end{equation*}
Hence,
\begin{align}
\abs{1-\tilde{t}_N}\leq \sum\limits_{k=\frac{N}{2}+1}^{N}\frac{\tilde{C}}{k^{2k}w_N}=\tilde{C}(N+1)^{2/3}\sum\limits_{k=\frac{N}{2}+1}^{N}\frac{1}{k^{2k}}
&\leq \tilde{C}_1(N+1)^{2/3}\frac{2^N}{N^{2N+1}}\nonumber\\
&\leq \tilde{C}_2\frac{2^N}{N^{N+\frac{1}{3}}}.\label{estimate:time:t_N}
\end{align}
Therefore,
\begin{equation}\label{estimate:L}
\abs{L(1)-\frac{2}{(N+1)^{2/3}}}\leq \tilde{C}\abs{1-\tilde{t}_N}w_N\leq \tilde{C}_3\frac{2^N}{N^{N+1}}.
\end{equation}
This completes the \descref{step-2a}{Step-2a}.

\descitem{Step-2b}{step-2b} Let $u^m$ be the entropy solution to \eqref{eq:discont} with data $u_0$ as in \eqref{def:counter-data} and flux $A^m$. In this step we prove that adapted entropy solution $u^m$ agrees with $u^N$ in some subset of $\R\times[0,1]$ for sufficiently large $m$. 

 Due to the fact that $f_m=g$ for $u\geq (N+1)^{-2/3}$, then characteristic with speed $2(N+1)^{-2/3}$ travels through $L_N(1)$ for $m\geq N+1$ and we observe that $u^m(x,t)=u^N(x,t)$ for $(x,t)\in D_N$ where $D_N$ is defined as (see figure \ref{figure:example-2})
$$D_N:= (-\f,0]\times[0,1]\cup \{(x,t),x\geq L_N(t),t\in[0,1]\}.$$
\end{description}

\descitem{Step-3}{step-C}\textit{Calculating the variation $\abs{u^N(a_m+,1)-u^N(a_m+,1)}$:} In this step we wish to calculate $\abs{u^N(a_m+,1)-u^N(a_m+,1)}$ where $u^N$ is the adapted entropy solution to \eqref{eq:discont} corresponding to flux $A^N$ and data $u_0$. 

Let characteristic corresponding to $u=\frac{1}{2m^{2/3}}$ and originated from $x=0$ moves through path $Q(t)$. Note that $Q(t)$ is analogous to $L(t)$. By a similar argument as in \eqref{estimate:L} we obtain 
\begin{equation*}
\abs{Q(1)-a_m}\leq \frac{\tilde{C}_2}{m^{2m+1}}.
\end{equation*}
Suppose $Q(t)$ hits $x=b_{m-1}$ and $x=a_m$ at time $t=\tilde{t}_1$ and $t=\tilde{t}_2$ respectively. Suppose $u^N(a_m-)=\tilde{w}_N$, then we have,
\begin{equation*}
a_m-b_{m-1}=2w_N(\tilde{t}_2-\tilde{t}_1)\mbox{ and }a_m-b_{m-1}=2\tilde{w}_N(1-\tilde{t}_1).
\end{equation*}
Since $\tilde{t}_2\geq1$ we get $w_N\leq \tilde{w}_N$. Then we have
\begin{equation}\label{estimate:reciprocal-diff}
\abs{\frac{1}{w_N}-\frac{1}{\tilde{w}_N}}\leq \frac{\tilde{t}_2-1}{a_m-b_{m-1}}.
\end{equation}
Note that
\begin{equation*}
a_m-b_{m-1}=\frac{1}{2n^{2/3}}-\frac{1}{2(n+1)^{2/3}}-\frac{1}{(n+1)^{2n+2}}\geq \frac{1}{3(n+1)^{5/3}}-\frac{1}{(n+1)^{2n+2}}\geq \frac{1}{6(n+1)^{5/3}}.
\end{equation*}
By a similar argument as in \eqref{estimate:time:t_N} we have
\begin{equation*}
\abs{\tilde{t}_2-1}\leq \frac{B}{m^{2m+1}}
\end{equation*}
Hence, from \eqref{estimate:reciprocal-diff}, we obtain
\begin{equation*}
1-\frac{w_N}{\tilde{w}_N}\leq \frac{24B m^{5/3}}{m^{2m+1}}\frac{1}{2m^{2/3}}\leq \frac{12B}{m^{2m}}.
\end{equation*}
Then, we get
\begin{equation*}
1-\frac{12B}{m^{2m}}\leq\frac{w_m}{\tilde{w}_m}.
\end{equation*}

For sufficiently large $m$, that is there exists an $m_0>1$ such that $1-\frac{12B}{m^{2m}}\geq \frac{1}{2}$ for $m\geq m_0$. Therefore, $\tilde{w}_m\leq 2{w}_m$. Hence, we get
\begin{equation*}
\abs{w_m-\tilde{w}_m}\leq 2w_m^2 \frac{\tilde{t}_2-1}{a_m-b_{m-1}}\leq \frac{24B}{m^{2m+\frac{2}{3}}}
\end{equation*}
Suppose $f_m(\tilde{v}_m)=g(\tilde{w}_m)$. Note that there exists $m_1>1$ such that $\abs{\tilde{w}_m-w_m}\leq \frac{1}{6m^{2/3}}$. Then we note that $\frac{1}{3m^{2/3}}\leq \tilde{w}_m\leq \frac{1}{2m^{2/3}}$. Subsequently, we get%We also note that
\begin{equation*}
f_m(\tilde{w}_m)=f_m\left(\frac{1}{2m^{2/3}}\right)=\frac{1}{4m^{4/3}}\left(1+\frac{1}{8m^{1/4}}\right)\geq \frac{1}{4m^{4/3}}=g(\tilde{w}_m)=f_m(\tilde{v}_m).
\end{equation*}
Thus we get $\tilde{w}_m\geq \tilde{v}_m$. Note that
\begin{equation*}
f_m\left(\frac{1}{3n^{2/3}}\right)=\frac{1}{9m^{4/3}}\left(1+\frac{4}{9m^{1/4}}\right)\leq \frac{2}{9m^{4/3}}\leq g(\tilde{w}_m)=f_m(\tilde{v}_m).
\end{equation*}
Hence $1/3\leq \tilde{v}_mm^{2/3}\leq 1/2$. Therefore,
\begin{equation*}
\tilde{w}_m^2=\frac{1}{4m^{4/3}}=f_m(\tilde{v}_m)=v_m^2\left(1+\frac{(1-m^{2/3}v_m)^3}{m^{1/4}}\right)\leq v_m^2\left(1+\frac{8}{27 m^{1/4}}\right).
\end{equation*}
Hence
\begin{equation*}
\abs{\tilde{w}_m-\tilde{v}_m}=\tilde{w}_m-\tilde{v}_m\geq \tilde{w}_m\left(\frac{\sqrt{1+\frac{8}{27m^{1/4}}}-1}{\sqrt{1+\frac{8}{27m^{1/4}}}}\right)\geq \frac{1}{27m^{1/4}}\left(\frac{1}{2m^{2/3}}-\frac{B}{m^{2m+\frac{2}{3}}}\right).
\end{equation*}
Therefore, we obtain for $N\geq 2m+1$,
\begin{equation}\label{lower-estimate-uN}
\abs{u^N(a_m+,1)-u^N(a_m-,1)}\geq \frac{1}{27m^{1/4}}\left(\frac{1}{2m^{2/3}}-\frac{B}{m^{2m+\frac{2}{3}}}\right).
\end{equation}
\descitem{Step-4}{step-D}\textit{Convergence of $u^m$ as $m\rr\f$:} In this step we want to pass to the limit for $\{u^m\}$ sequence as $m\rr\f$. Then we show that limiting function is an adapted entropy solution to \eqref{eq:discont} for initial data $u_0$ defined as in \eqref{def:counter-data}.

 Note that $\norm{u^N-u^m}_{L^1(\R\times[0,1])}\leq \tilde{B}N^{-2/3}$ for all $m\geq N$. Therefore, $\{u^m\}_{m\geq1}$ is a Cauchy sequence in $L^1(\R\times[0,1])$. Hence, there exists a $u\in L^1(\R\times[0,1])$ such that $u^m\rr u$ in $L^1(\R\times[0.1])$. We also note that $A^N(x,\cdot)=A(x,\cdot)$ for $x\in\R\setminus[0,b_N]$. This implies $k^\pm_{\al,N}(x)=k^\pm_\al(x)$ for $x\in \R\setminus[0,b_N]$ and we get $k_{\al,N}^\pm\rr0$ as $N\rr\f$. Since $u^N$ is adapted entropy solution, $u^N\rr u$ in $L^1(\R\times[0,T])$, $k^\pm_{\al,N}\rr k^\pm_{\al}, A^N(x,u)\rr A(x,u)$ for a.e. $x\in\R$, we get $u$ is adapted entropy solution to \eqref{eq:discont} for data $u_0$ as in \eqref{def:counter-data} with flux $A(x,u)$.
 \descitem{Step-5}{step-E}\textit{BV blow up:} Let $u$ be the adapted entropy solution obtained in \descref{step-D}{Step-4}. By using \descref{step-2b}{Step-2b}, we note that $u(x,t)=u^N(x,t)$ for $(x,t)\in D^N$. From \eqref{lower-estimate-uN} we get
\begin{equation*}
\abs{u(a_m+,1)-u(a_m-,1)}\geq \frac{1}{27m^{1/4}}\left(\frac{1}{2m^{2/3}}-\frac{B}{m^{2m+\frac{2}{3}}}\right).
\end{equation*}
Hence we have $TV(u(\cdot,1))=\f$.
\end{description}

\begin{remark}
Let $u_0$ be defined as in \eqref{def:counter-data}. Due to the finite speed of propagation we can consider a data $\tilde{u}_0=u_0\chi_{[-M,M]}$ for some large $M>0$ and solution still gives a TV blow up at time $t=1$.
\end{remark}

	\noindent\textbf{Acknowledgement.} The first and third author would like to thank Department of Atomic Energy, Government of India, under project no. 12-R\&D-TFR-5.01-0520 for support. The first author acknowledges the Inspire faculty-research grant DST/INSPIRE/04/2016/000237.


\section{References}
 \begin{thebibliography}{99}
 	

 	
 	\bibitem{ADGG}
 	\newblock  Adimurthi, R. Dutta, S. S. Ghoshal and G. D. V. Gowda. 
 	\newblock Existence and nonexistence of TV bounds for scalar conservation laws with discontinuous flux. 
 	\newblock {\em Comm. Pure Appl. Math.}, 64 (2011), 1, 84--115.
 	
 	\bibitem{AG}
 	\newblock Adimurthi and  S. S. Ghoshal,
 	\newblock Exact and optimal controllability for scalar conservation laws with discontinuous flux.
 	\newblock Preprint, (2020).
 	
 	\bibitem{AJG}
 	\newblock Adimurthi, J. Jaffr\'e and G. D. Veerappa Gowda, 
 	\newblock Godunov-type methods for conservation laws with a flux function discontinuous in space,
 	\newblock{\em SIAM J. Numer. Anal.} 42, (2004), no. 1, 179--208.
 	 	
	\bibitem{AMV}
 	\newblock Adimurthi, S. Mishra and G. D. Veerappa Gowda, 
 	\newblock Optimal entropy solutions for conservation laws with discontinuous flux functions,
 	\newblock{\em J. Hyperbolic Differ. Equ.} 2, (2005), 783--837.
 	
 	\bibitem{mishra2005}
 	\newblock Adimurthi, S. Mishra, and G. D. Veerappa Gowda,  
 	\newblock Convergence of upwind finite difference schemes for a scalar conservation law with indefinite discontinuities in 	the flux function,
 	\newblock  {\em SIAM J. Numer. Anal.}, 43, (2005), 2, 559--577.
 	
 	\bibitem{mishra2007convergence}
 	\newblock  Adimurthi, S. Mishra, and G. D. Veerappa Gowda, 
 	\newblock Convergence of Godunov type methods for a conservation law with a spatially varying discontinuous flux function.
 	\newblock {\em Math. Comp.}, 76, (2007),  259, 1219--1242.
 	
 	
 	
 	\bibitem{adimurthi2000conservation}
 	\newblock Adimurthi and G. D. Veerappa Gowda,  
 	\newblock Conservation law with discontinuous flux,
 	\newblock {\em J. Math. Kyoto Univ.},  43, (2003), 1, 27--70.
 	
 		
 	\bibitem{andreianov2013vanishing}
 	\newblock B. Andreianov and C. Canc\'es,  
 	\newblock Vanishing capillarity solutions of buckley–leverett equation with gravity in two-rocks medium,
 	\newblock {\em Comput. Geosci.}, 17, (2013), 3, 551--572.
 	
 	\bibitem{AKR}
 	\newblock B. Andreianov, K. H. Karlsen and N. H. Risebro,
 	\newblock A theory of $L^1$-dissipative solvers for scalar conservation laws with discontinuous flux.
 	\newblock{\em Arch. Ration. Mech. Anal.} 201, 1, 27--86, 2011.
 	

 	\bibitem{AudussePerthame} 
 	\newblock  E. Audusse and B. Perthame,  
 	\newblock Uniqueness  for scalar conservation laws with discontinuous flux via adapted entropies,  
 	\newblock{\em Proc. Roy. Soc. Edinburgh Sect. A} 135, (2005), 253--265. 
 	
 		\bibitem{Bressan}
 	\newblock A. Bressan, 
 	\newblock Hyperbolic systems of conservation laws. The one-dimensional Cauchy problem,
 	\newblock\textit{Oxford Lecture Series in Mathematics and its Applications, 20. Oxford University Press, Oxford}, 2000. xii+250 pp.
 	
 	\bibitem{BGKT}
 	\newblock R. B\"urger, A. Garc{\'\i}a, K. Karlsen and J. Towers,
 	\newblock A family of numerical schemes for kinematic flows with discontinuous flux, 
 	\newblock {\em J. Eng. Math.}, 60, (2008), (3-4), 387--425.
 	
 	\bibitem{burger2006extended}
 	\newblock R. B\"urger, A. Garcia, K. H. Karlsen, and J. D. Towers,  
 	\newblock On an extended clarifier-thickener model  with  singular  source  and  sink  terms.
 	\newblock {\em  European J. Appl. Math.}, 17 (2006), 3, 257--292.%European  Journal  of  Applied  Mathematics, 42817(3):257--292, 2006.
 	
 	\bibitem{bkrt2} 
 	\newblock R.~B\"{u}rger, K. H.~Karlsen, N. H.~Risebro and J.D.~Towers, 
 	\newblock Well-posedness in $BV_t$ and convergence of a  difference scheme for continuous sedimentation in ideal clarifier-thickener units,  
 	\newblock {\em Numer.\ Math.},  97, (2004), 1, 25--65. 
 	
 	\bibitem{BKT}
 	\newblock R. B\"urger, K. H. Karlsen and J. D. Towers,
 	\newblock An Engquist-Osher-type scheme for conservation laws with discontinuous flux adapted to flux connections.
 	\newblock{\em SIAM J. Numer. Anal.} 47, (2009), 3, 1684--1712.
 	
	\bibitem{CEK}
	\newblock G. Q. Chen, N. Even, C. Klingenberg,
	 \newblock Hyperbolic conservation laws with discontinuous fluxes and hydrodynamic limit for particle systems. 
	 \newblock {\em J. Differ. Equ.} 245(11), (2008), 3095--3126. 
 	
 	%\bibitem{CranMaj:Monoton} 
 	%\newblock M. G.~Crandall and A.~Majda,	
 	%\newblock Monotone difference approximations for scalar conservation laws. 
 	%\newblock {\em Math. Comp.},  34, (1980), 1--21.
 	
 	 \bibitem{diehl1996conservation}
 	\newblock S.~ Diehl,
 	\newblock A conservation law with point source and discontinuous flux function modelling continuous sedimentation,
 	\newblock {\em SIAM J. Appl. Math.}, 56 (1996), 2, 388--419.
 	
 	 	
    \bibitem{garavello2007conservation} 
 	\newblock M. Garavello, R. Natalini, B. Piccoli, and A. Terracina. 
 	\newblock Conservation laws with discontinuous flux.
 	\newblock {\em  Netw. Heterog. Media}, 2 (2007), 1, 159--179.
 	
 	
 	
 	\bibitem{GJT_2019}
 	\newblock S. S. Ghoshal, A. Jana, and J. D. Towers,
 	\newblock Convergence of a Godunov scheme to an Audusse-Perthame 
 	adapted entropy solution for conservation laws with BV spatial flux,
 	\newblock {\em Numer. Math.} (2020). https://doi.org/10.1007/s00211-020-01150-y%Preprint https://arxiv.org/pdf/2003.10321.pdf
 	
 	\bibitem{GTV-2020}
 	\newblock S. S. Ghoshal,  J. D. Towers and G. Vaidya,
 	\newblock Numerical approximation of conservation laws with degeneracies and spatial heterogeneities.
 	\newblock Preprint, 2020.
 	
   	
   \bibitem{Ghoshal-JDE}
   \newblock S. S. Ghoshal,
   \newblock Optimal results on TV bounds for scalar conservation laws with discontinuous flux,
   \newblock {\em J. Differential Equations}, 258 (2015), 3, 980--1014.
   
   %\bibitem{NHM} 
   %\newblock S. S. Ghoshal, 
   %\newblock BV Regularity Near The Interface For Nonuniform Convex Discontinuous Flux, 
   %\newblock {\em Netw. Heterog. Media}, 11, 2, (2016), 331--348. 
   
%   
%
% 	\bibitem{Risebro1} 
% 	\newblock K. Karlsen, S. Mishra, and N. Risebro.  
% 	\newblock Convergence of finite volume schemes for triangular systems of conservation laws.
% 	\newblock Numer. Math., 111(4) 559--589, 2009.
% 	
 	
 	\bibitem{karlsen2004convergence}
 	\newblock K.  H.  Karlsen  and  J.  D.  Towers,   
 	\newblock Convergence  of  the  Lax-Friedrichs  scheme  and  stability for  conservation  laws  with  a  discontinuous  	space-time  dependent  flux,
 	\newblock {\em Chinese Ann. Math. Ser. B}, 25 (2004), 3, 287--318. %}Chinese  Annals  of Mathematics, 25(03):287--318, 2004.
 	
 	\bibitem{KT}
 	\newblock K. H. Karlsen and J. D. Towers,
 	\newblock Convergence of a Godunov scheme for conservation laws with a discontinuous flux lacking the crossing condition.
 	\newblock{\em J. Hyperbolic Differ. Equ.} 14, (2017), 4, 671--701.
 	
	%\bibitem{leveque_book}
	%\newblock R. J. Leveque, 
	%\newblock Finite volume methods for hyperbolic problems,
	%\newblock Cambridge University Press, Cambridge, UK, 2002.
		
	 \bibitem{Panov2009a}
	\newblock E. Y. Panov,
	\newblock On existence and uniqueness of entropy solutions to the Cauchy problem for a conservation law with discontinuous flux,
	\newblock  {\em J. Hyperbolic Differ. Equ.},  6, (2009), 3, 525--548.	
	
 	\bibitem{piccoli2018general}
 	\newblock B. Piccoli and M. Tournus,  
 	\newblock A general bv existence result for conservation laws with spatial heterogeneities.
 	\newblock {\em SIAM J. Math. Anal.}, 50, (2018), 3, 2901--2927.%}SIAM Journal on Mathematical Analysis, 50(3):2901--2927, 2018.
	
 	
 	
 	\bibitem{shen2017uniqueness}
 	\newblock W.~ Shen, 
 	\newblock On the uniqueness of vanishing viscosity solutions for riemann problems for polymer flooding,
 	\newblock {\em NoDEA Nonlinear Differential Equations Appl.} 24 (2017), 4, 37, 25 pp.%}Nonlinear Differential Equations and Applications NoDEA, 24(4):37, 2017
 	
 	
 	\bibitem{towers_disc_flux_1}
 	\newblock J. D. Towers, 
 	\newblock Convergence of a difference scheme for conservation laws with a discontinuous flux, 
 	\newblock {\em SIAM J. Numer. Anal.} 38 (2000),  681--698.
 	
 	\bibitem{JDT_2020} 
 	\newblock  J. D. Towers, 
 	\newblock An existence result for conservation laws having BV spatial flux heterogeneities - without concavity, 
 	\newblock J. Differ. Equ. 269 (2020), 5754--5764.
 \end{thebibliography}
\end{document}